\documentclass{article}
\usepackage[a4paper,top=2cm,bottom=2cm,left=3cm,right=3cm,marginparwidth=1.75cm]{geometry}     
\usepackage{fancyhdr}       
\usepackage{parskip}        
\usepackage{titlesec}       
\usepackage{tikz}           
\usetikzlibrary{shapes.geometric}
\usepackage{pgfplots}       
\usepackage{pgfplotstable}  
\usepackage{xcolor}         
\usepackage{amsmath}        
\usepackage{siunitx}        
\usepackage{listings}       
\PassOptionsToPackage{hyphens}{url} 
\usepackage{hyperref}       

\usepackage{cleveref}       
\usepackage{textcomp}       
\usepackage[english]{babel} 
\usepackage[labelfont=bf]{caption} 
\usepackage{subcaption}     
\usepackage{csquotes}       
\usepackage{verbatim}
\usepackage[backend=biber,style=numeric-comp,sorting=none,maxbibnames=99]{biblatex} 
\bibliography{bibliography.bib}

\makeatletter
\newcommand\setmaxbibnames[1]{\renewcommand\blx@maxbibnames{#1}}
\makeatletter

\usepackage{tabularx}       
\usepackage{booktabs}       
\newcolumntype{Y}{>{\centering\arraybackslash}X}
\usepackage{float}          
\usepackage{calc}           
\usepackage{blindtext}

\usepackage[version=4]{mhchem}
\usepackage{textcomp, gensymb}
\usepackage{amssymb, amsthm, amsfonts}
\usepackage{pdfpages}
\usepackage{plimsoll}
\usepackage{longtable}
\usepackage{graphicx}
\usepackage[inkscapelatex=false]{svg}
\usepackage{mathtools}

\newtheorem{lem}{Lemma}
\newtheorem{theorem}{Theorem}
\newtheorem{cor}{Corollary}
\newtheorem{assum}{Assumption}
\newtheorem{remark}{Remark}

\DeclareMathOperator{\sign}{sign}
\DeclareMathOperator{\sat}{sat}

\usepackage{authblk}
\title{\textbf{Graph-Convolutional Reinforcement Learning for Finite-Time Stabilization of a Chaotic Chua Circuit via Unified Flexible Terminal Sliding-Mode Control}}

\author[1,*]{Hooman Fatoorehchi}
\author[2]{Majid Roohi}
\author[1]{Rasmus Kold}
\author[3]{Matthias Ehrhardt}
\author[2]{Andreas Basse-O’Connor}
\author[1]{Sebastian Nis Bay Villadsen}
\author[1]{Philip L. Fosbøl}
\affil[1]{Center for Energy Resources Engineering (CERE), Department of Chemical and Biochemical Engineering, Technical University of Denmark, Søltofts Plads Building 228A, Kongens Lyngby, 2800, Denmark.}

\affil[2]{Department of Mathematics, Aarhus University, 8000 Aarhus, Denmark.}

\affil[3]{Chair of Applied and Computational Mathematics, University of Wuppertal, Wuppertal, Germany.}

\affil[ ]{$^*$\textit{Corresponding author.} 
Email: hoofa@dtu.dk}
\date{}
\begin{document}
\maketitle
\begin{abstract}
This paper introduces a graph convolutional reinforcement learning framework for the finite-time stabilization of chaotic dynamics in a physical Chua circuit. 
The circuit is modeled directly in terms of capacitor voltages, inductor current, and physical component parameters. 
This preserves the connection between the mathematical model and the underlying electrical system. 
A unified, flexible terminal sliding-mode controller is developed to achieve stabilization within a finite time in the presence of uncertainties in the lumped model, external disturbances, and bounded control inputs. 
The controller incorporates state-dependent nonlinear exponents to provide flexible convergence characteristics while avoiding singular terms near equilibrium. 
To minimize dependence on manual controller tuning, the admissible controller parameters are optimized using an actor-critic reinforcement learning scheme with a graph convolutional network that leverages the circuit dynamics' coupling structure. 
The learned parameters are constrained to the ranges required by the theoretical stability conditions. Numerical experiments demonstrate the rapid suppression of chaotic oscillations and the finite-time convergence of the state and sliding variables while respecting the prescribed input-saturation limits. 
Compared to a reference control method, the proposed approach reduces the IAE, ISE, ITAE, and ITSE by 56.27\,\%, 56.40\,\%, 61.69\,\%, and 64.70\,\%, respectively.
These results suggest that graph-based reinforcement learning can effectively tune parameters while retaining the robustness and analytical structure of the underlying sliding-mode controller.

\end{abstract}
\section{Introduction}
Chaos is a distinctive phenomenon in nonlinear dynamical systems, characterized by deterministic yet highly irregular behavior and strong sensitivity to initial conditions.
Since the pioneering observations of Lorenz, chaotic dynamics have attracted considerable attention across mathematics, physics, and engineering. 
Their complex evolution and limited long-term predictability have made the analysis and characterization of chaotic systems an important area of nonlinear dynamics \cite{rasband2015chaotic}.
The distinctive properties of chaotic dynamics have also been exploited in a wide range of practical applications, including secure communication and encryption, signal processing, electronic circuits, robotics, power systems, and chemical processes. 
While chaotic behavior can be deliberately exploited in applications such as information security and signal generation, 
it may also lead to undesirable oscillations and degraded performance in engineering systems. 
This dual role has motivated extensive research on both the utilization and effective control of chaotic dynamics \cite{jaiswal2024}.
Chua's circuit remains one of the most extensively studied physical realizations of chaos, particularly because of its simple electronic structure and rich nonlinear behavior, including the well-known double-scroll attractor \cite{chua1994a, kennedy1992a}.
Its dynamics are commonly represented through normalized state-space equations, which provide a compact mathematical description and facilitate nonlinear analysis. 
However, normalization partly removes the direct correspondence between the model variables and the electrical quantities of the original circuit, making the influence of individual circuit components less transparent.

A circuit-level representation offers a different perspective by preserving the actual electrical states and component parameters. 
In this formulation, the capacitor voltages $V_{C_1}$ and $V_{C_2}$ and the inductor current $i_L$ remain the system states, while $C_1$, $C_2$, $L$, and $R_P$ explicitly appear in the governing equations. The nonlinear resistor is also described through its physical slopes and breakpoint voltage. 
This representation therefore maintains a direct link between the observed chaotic dynamics and the underlying circuit configuration, allowing the influence of component values and parameter variations on the system behavior to be examined more explicitly.
Stabilizing chaotic dynamics is a challenging control problem because strong nonlinearities and sensitivity to perturbations can significantly affect both transient and steady-state behavior. 
Various strategies, including Ott-Grebogi-Yorke (OGY) control, active control, adaptive and backstepping schemes, model predictive control, intelligent control methods, and sliding mode control, have been developed to regulate chaotic systems toward prescribed equilibria or trajectories \cite{Truong, Agha1, agha2, Azioune}.
Their practical effectiveness, however, remains closely tied to robustness against uncertain dynamics and disturbances, convergence characteristics, model dependence, and constraints imposed on the control input.

\textit{Sliding mode control} (SMC) provides a robust framework for regulating nonlinear and chaotic dynamics through the construction of an invariant sliding manifold.
Its low sensitivity to matched uncertainties and disturbances has supported its extensive use in chaos stabilization and synchronization \cite{SMC}.
In recent years, SMC has received increasing attention, with various control schemes developed to address different nonlinear systems and engineering applications. 
For example, in \cite{10403641}, an adaptive SMC reaching law was developed for 
PMSM speed regulation, with the main focus on reaching speed and chattering reduction. 
Gedefaw, Abdissa, and Lemma \cite{gedef}
developed a super-twisting SMC with a fuzzy PID surface for quadrotor trajectory tracking, where fuzzy logic was employed for online tuning of the controller gains.
In \cite{SAK}, a high-order SMC tuned by reinforcement learning was designed for a three-level NPC inverter, enabling real-time gain adjustment under varying grid conditions and LVRT operation.
An adaptive SMC combined with radial basis function (RBF) neural networks was designed in \cite{FENG} for robotic excavators, using the neural network to estimate and compensate for unknown system dynamics.
For mobile manipulator systems, Ref.~\cite{meny} proposed a backstepping fuzzy SMC whose control parameters were optimized using PSO for simultaneous end-effector and platform motion tracking.
Ref.~\cite{ram} investigated two SMC schemes for the chaotic Belousov–Zhabotinsky (BZ) reaction, employing smoothed and adaptive fuzzy structures to address tracking performance and chattering.
In \cite{lan}, a data-driven SMC framework was developed for MIMO nonlinear systems with partially unknown dynamics, combining a robust reaching controller with an SDP-based nominal controller.

Despite these developments, some limitations remain in the existing SMC literature: 
(1) the performance of several SMC schemes is highly dependent on controller parameter selection, which may require considerable tuning effort; 
(2) actuator saturation is not explicitly incorporated into the design of some existing controllers, despite its practical importance in physical systems; 
(3) certain SMC approaches still rely on sufficiently accurate knowledge of the system dynamics, which may restrict their applicability when model information is incomplete; 
and (4) although intelligent optimization techniques have increasingly been combined with SMC, the use of deep and graph-based learning for controller parameter optimization remains relatively limited.

To address these issues, this study develops a graph-learning-based \textit{unified flexible terminal sliding mode control} (UFT-SMC) framework for the finite-time stabilization of chaotic dynamics. 
Unlike the commonly adopted normalized Chua models, the controller is implemented on a circuit-level representation in which the capacitor voltages, inductor current, and physical circuit parameters are retained explicitly. 
The control formulation accounts for lumped uncertainties, bounded disturbances, and input saturation, while a GCN-based learning scheme is employed to optimize the controller parameters based on their interdependencies and closed-loop performance.

Accordingly, the main contributions are summarized as follows: 
(1) a physical Chua circuit model is considered to preserve the direct connection between chaotic dynamics and actual electrical quantities; 
(2) a UFT-SMC scheme is developed for finite-time stabilization under uncertain dynamics and constrained control inputs; 
(3) a GCN-based framework is introduced for systematic optimization of the controller parameters; 
and (4) the proposed approach is assessed through convergence, control effort, integral performance indices, and comparative control studies.

The remainder of this paper is organized as follows. 
Section~\ref{sec:chua} presents the physical Chua circuit model and its dynamical characteristics. Section~\ref{sec:3} develops the proposed UFT-SMC and establishes the corresponding stability and finite-time convergence results. 
Section~\ref{sec:stochastic} presents the stochastic robustness analysis of the proposed control method.
Section~\ref{sec:5} introduces the graph-learning-based parameter optimization framework. 
Section~\ref{sec:6} presents the numerical results and comparative evaluations. 
Finally, Section ~\ref{sec:7} concludes the paper. 

\section{The Chua circuit}\label{sec:chua}

This section presents the Chua circuit model and its physical parameters, 
followed by an analysis of its dynamical behavior and Lyapunov exponents.
\subsection{Circuit Description and Mathematical Model}
Consider the physical Chua circuit with capacitor voltages $V_{C_1}$ and $V_{C_2}$ and inductor current $i_L$ as its state variables. The corresponding circuit-level dynamics are expressed as \cite{kennedy1992a}
\begin{equation}\label{Chua_1}
\left\{
\begin{aligned}
C_1\dot{V}_{C_1} &= \frac{1}{R_P} (V_{C_2}-V_{C_1})-f(V_{C_1}),\\
C_2\dot{V}_{C_2} &= \frac{1}{R_P} (V_{C_1}-V_{C_2})+i_L,\\
L\dot{i}_L &= -V_{C_2}.
\end{aligned}
\right.
\end{equation}
Here, $C_1$ and $C_2$ denote the capacitances, $L$ is the inductance, and $R_P$ is the variable resistance. 
The nonlinear current-voltage characteristic is described by
\begin{equation*}
    f(V_{C_1}) = m_0 V_{C_1} + \frac{1}{2}(m_1-m_0) \bigl( |V_{C_1}+B_P| - |V_{C_1}-B_P| \bigr),
\end{equation*}
where $m_0$ and $m_1$ are the slopes of the linear regions and $B_P$ denotes the breakpoint voltage.

The physical realization of the circuit is shown in Figure~\ref{fig:diagramSetup}. 
The nonlinear resistor $N_R$, corresponding to Chua's diode, is implemented using two operational amplifiers.

\begin{figure}[htb!]
    \centering
    \includegraphics[width=0.9\textwidth]{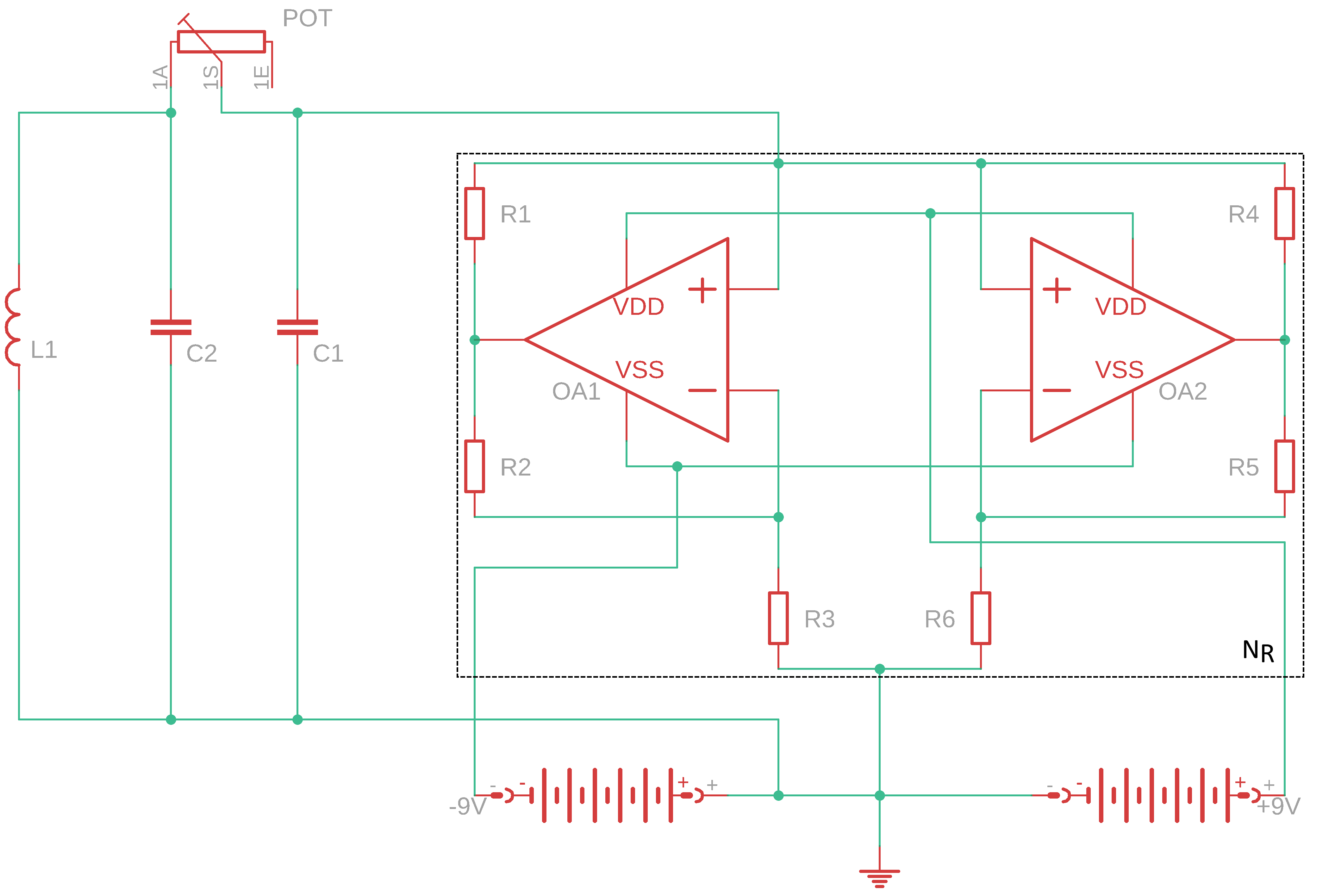}
    \caption{Electrical schematic of the physical Chua circuit.}
    \label{fig:diagramSetup}
\end{figure}

Depending on the resistor configuration, $N_R$ can exhibit three or five linear regions within the voltage range defined by the operational-amplifier saturation level $E_{\sat}$. For $R_1=R_2$, $R_4=R_5$, and sufficiently large operational-amplifier gain, the breakpoint voltages are given by
\begin{equation*}
   \pm B_P^+ = \frac{R_3}{R_2+R_3}E_{\sat}, \qquad
   \pm B_P^- = \frac{R_6}{R_5+R_6}E_{\sat}.
\end{equation*}
When $R_2=R_5$ and $R_3=R_6$, the breakpoint levels coincide, resulting in a three-region piecewise-linear characteristic. 
The nominal component values of the circuit are listed in Table~\ref{tab:components}.

\begin{table}[ht!]
    \centering
        \caption{Nominal component values of the Chua circuit.}
    \label{tab:components}
    \begin{tabular}{@{}llll@{}}
    \toprule
    Component & Description & Value & Tolerance \\ \midrule
    $POT$ & Potentiometer & 2.5 k$\Omega$ & $\pm5\%$ \\
    $C_1$ & Capacitor & 10 nF & $\pm5\%$ \\
    $C_2$ & Capacitor & 100 nF & $\pm5\%$ \\
    $L$ & Inductor & 20 mH & $\pm5\%$ \\
    $OA_1$ & Operational amplifier & -- & $\pm5\%$ \\
    $R_1$ & Resistor & 220 $\Omega$ & $\pm5\%$ \\
    $R_2$ & Resistor & 220 $\Omega$ & $\pm5\%$ \\
    $R_3$ & Resistor & 2.2 k$\Omega$ & $\pm5\%$ \\
    $OA_2$ & Operational amplifier & -- & $\pm5\%$ \\
    $R_4$ & Resistor & 22 k$\Omega$ & $\pm5\%$ \\
    $R_5$ & Resistor & 22 k$\Omega$ & $\pm5\%$ \\
    $R_6$ & Resistor & 3.3 k$\Omega$ & $\pm5\%$ \\ \bottomrule
    \end{tabular}
\end{table}

\subsection{Chaotic Dynamics and Lyapunov Analysis}
The dynamical behavior of the Chua circuit is strongly influenced by the variable resistance $R_P$. 
For the component values listed in Table~\ref{tab:components}, the circuit exhibits irregular oscillatory behavior at $R_P=1800\,\Omega$. 
The corresponding state responses and phase-space trajectory are shown in Figure~\ref{fig}. 
Figures~\ref{fig}(a)--(c) present the time responses of $V_{C_1}$, $V_{C_2}$, and $i_L$, respectively, 
while Figure~\ref{fig}(d) shows the resulting three-dimensional attractor. 
The double-scroll structure observed experimentally using an oscilloscope is shown in Figure~\ref{fig}(e).
The experimentally observed attractor is used here to validate the chaotic operating regime of the physical circuit. 
The subsequent closed-loop control results are obtained numerically; experimental closed-loop validation is beyond the scope of the present study.


\begin{figure}[ht!]
\centering
\includegraphics[width=0.82\linewidth]{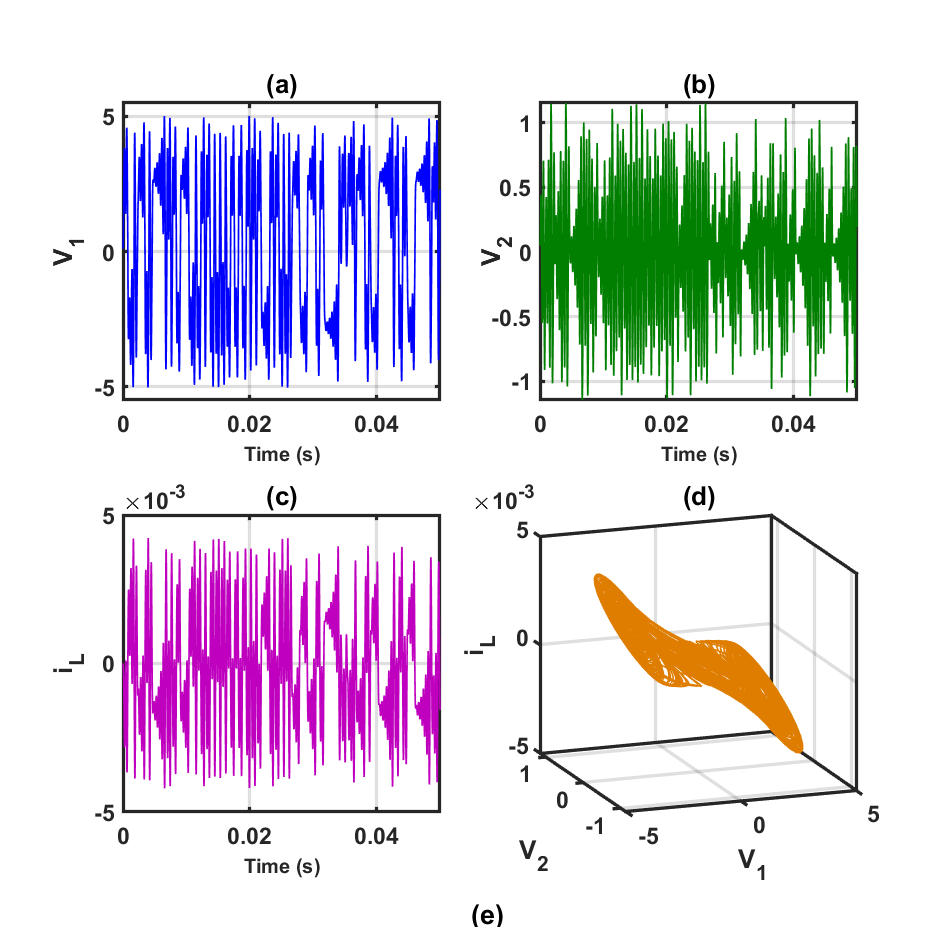}\\[-1pt]
\includegraphics[width=0.7\linewidth]{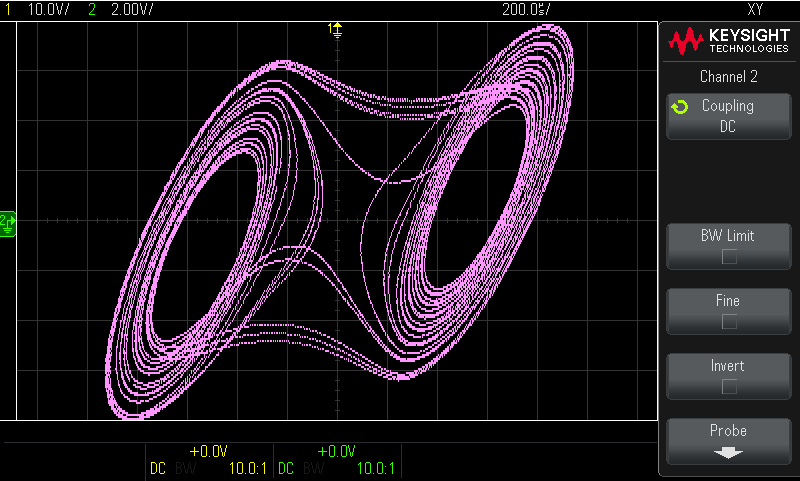}
\caption{Dynamical behavior of the Chua circuit at $R_P=1800\,\Omega$: (a)--(c) time responses of $V_{C_1}$, $V_{C_2}$, and $i_L$, respectively; (d) three-dimensional phase-space trajectory; and (e) experimentally observed double-scroll attractor.}
\label{fig}
\end{figure}

To quantitatively characterize the observed chaotic behavior, the Lyapunov
exponents are evaluated from the evolution of infinitesimal perturbations along
the system trajectory. For a nonlinear system 
    $\dot{\mathbf{x}}=\mathbf{F}(\mathbf{x})$,
the perturbation dynamics are governed by
\begin{equation}
    \delta\dot{\mathbf{x}} = J(\mathbf{x})\,\delta\mathbf{x},\qquad
    J(\mathbf{x}) = \frac{\partial \mathbf{F}}{\partial \mathbf{x}},
\end{equation}
where $J(\mathbf{x})$ denotes the Jacobian matrix.
The \textit{Lyapunov exponents} are defined as
\begin{equation}
    \lambda_i = \lim_{t\to\infty}
    \frac{1}{t}\ln\biggl(\frac{\|\delta\mathbf{x}_i(t)\|}
         {\|\delta\mathbf{x}_i(0)\|}\biggr),\qquad i=1,2,3.
\end{equation}

Because the circuit is formulated in physical units, the Lyapunov exponents are reported on the dimensionless time scale adopted in the numerical analysis. 
Let $\tau = t/T_0$ denote the dimensionless time associated with the numerical characterization, where $T_0$ is the chosen reference time scale.
If $\lambda_i$ are the Lyapunov exponents with respect to physical time $t$, then the corresponding exponents with respect to $\tau$ are
    $\Lambda_i=T_0\lambda_i$, $i=1,2,3$,
i.e.\ only the Lyapunov analysis uses the normalized time coordinate, while the control model remains in physical variables.

At $R_P=1800~\Omega$, the resulting dimensionless Lyapunov spectrum is
\begin{equation}
    \Lambda_1=0.43651023,\qquad
    \Lambda_2=0.00017188,\qquad
    \Lambda_3=-3.35794381.
\end{equation}
The presence of one positive exponent, together with an exponent close to zero and a negative exponent,
is consistent with the chaotic behavior observed in the phase-space trajectories.

The evolution of the dimensionless Lyapunov spectrum and the bifurcation behavior are shown in Figure~\ref{fig:LE_bifurcation}. 
As seen in Figure~\ref{fig:LE_bifurcation}(a), the three exponents settle after the initial transient, with $\Lambda_1$ remaining positive. 
Figure~\ref{fig:LE_bifurcation}(b) shows the local maxima of $V_{C_1}$ as $R_P$ is varied, revealing changes in the long-term dynamics of the circuit. 
The dashed line marks $R_P=1800\,\Omega$, which is used as the operating point in the subsequent control study.

\begin{figure}[ht!]
    \centering
    \includegraphics[width=\linewidth]{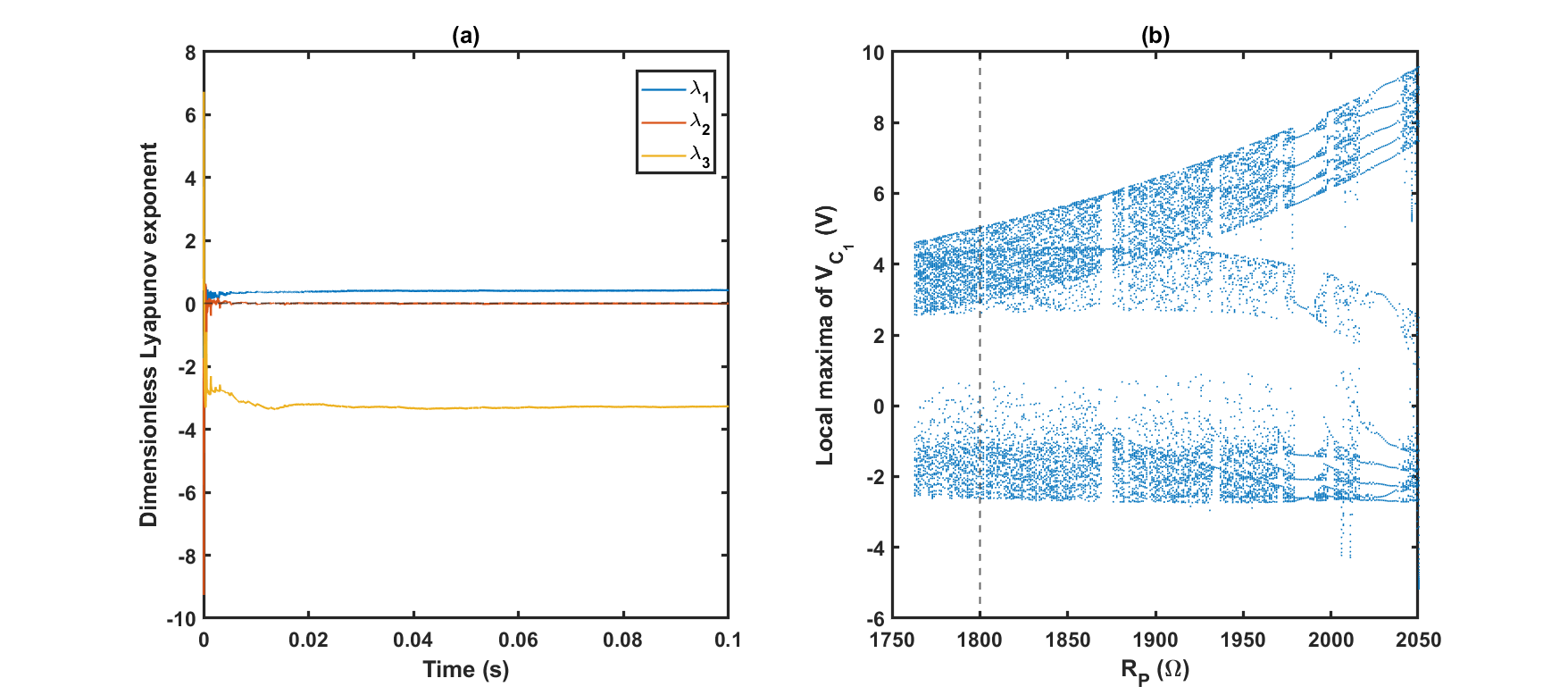}
    \caption{Dynamical characterization of the Chua circuit:
    (a) evolution of the dimensionless Lyapunov exponents at
    $R_P=1800~\Omega$; (b) bifurcation diagram based on the local maxima
    of $V_{C_1}$ with respect to $R_P$. The dashed line in (b) marks the
    operating point $R_P=1800~\Omega$.}
    \label{fig:LE_bifurcation}
\end{figure}

\section{Design of the control method}\label{sec:3}
Here, the Chua circuit is first reformulated into a control-orientated state-space representation accounting for practical system effects. 
The proposed \textit{unified flexible terminal sliding mode control} (UFT-SMC) is then developed for this formulation.
\subsection{The control system}
To facilitate the controller design, the Chua chaotic circuit is reformulated into a state-space representation by defining
\begin{equation*}
    x_1=V_{C_1}, \qquad x_2=V_{C_2}, \qquad x_3=i_L,
\end{equation*}
where $x_1$ and $x_2$ denote the voltages across capacitors $C_1$ and $C_2$, respectively, and $x_3$ represents the inductor current. 
Accordingly, the circuit equations can be rewritten as
\begin{equation} \label{Chua_2}
\left\{
\begin{aligned}
  \dot{x}_1 &= a(x_2-x_1)-\frac{1}{C_1}f(x_1),\\
  \dot{x}_2 &= b(x_1-x_2)+cx_3,\\
  \dot{x}_3 &= -dx_2. 
\end{aligned}
\right.
\end{equation}
where
\begin{equation*}
  a=\frac{1}{R_PC_1}, \qquad b=\frac{1}{R_PC_2}, \qquad
   c=\frac{1}{C_2}, \qquad d=\frac{1}{L},
\end{equation*}
and
\begin{equation*}
    f(x_1) = m_0 x_1 + \frac{1}{2} (m_1 - m_0)\, (|x_1 + B_P|-|x_1-B_P|).
\end{equation*}
The controlled Chua system \eqref{Chua_2}, accounting for input saturation and lumped uncertainties, is written as
\begin{equation}\label{system}
    \dot{x}_i = g_i(x) + \eta_i(x,t) + \phi(u_i), \qquad i=1,2,3,
\end{equation}
where
\begin{equation*}
  g_1(x)=a(x_2-x_1)-\frac{1}{C_1}f(x_1), \qquad 
  g_2(x)=b(x_1-x_2)+cx_3, \qquad
   g_3(x)=-dx_2.
\end{equation*}
Here, $\eta_i(x,t)=\Delta g_i(x,t)+d_i(t)$ denotes the \textit{lumped uncertainty}, incorporating model uncertainties and external disturbances. 
Furthermore, $u_i$ denotes the control input,
while the \textit{saturated control signal} is defined as
\begin{equation}\label{sat}
   \phi(u_i)=\sat(u_i),
\end{equation}
where
\begin{equation*}
\sat(u_i(t))=
\begin{cases}
u_i^l, & u_i(t)\le u_i^l,\\
u_i(t), & u_i^l<u_i(t)<u_i^r,\\
u_i^r, & u_i(t)\ge u_i^r,
\end{cases}
\qquad i=1,2,3.
\end{equation*}
Here, $u_i^r\in\mathbb{R}^+$ and $u_i^l\in\mathbb{R}^-$ denote the
upper and lower saturation limits, respectively. 
The \textit{saturation-induced input mismatch} is defined as
\begin{equation}\label{sat_error}
\Delta u_i(t)=\sat(u_i(t))-u_i(t),
\end{equation}
and hence
\begin{equation}\label{sat_decomp}
\phi(u_i)=u_i+\Delta u_i.
\end{equation}
\begin{assum}
 The states of the Chua chaotic system evolve within a bounded attractor in the phase space, as commonly observed in chaotic dynamics \cite{curran1997absolute, fradkov2005control}. 
 Therefore, the nonlinear functions $g_i(x)$ are assumed to be bounded, i.e., there exist positive constants $\xi_i$ such that
\begin{equation*}
   |g_i(x)| \le \xi_i, \qquad i=1,2,3.
\end{equation*}
Moreover, the lumped uncertainty term $\eta_i(x,t)$, which includes both model uncertainties and external disturbances, is assumed to be bounded. Hence, there exist positive constants $k_i$ such that
\begin{equation}\label{limit_sys_2}
   |\eta_i(x,t)| \le k_i, \qquad i=1,2,3.
\end{equation}
Therefore,
\begin{equation} \label{limit_sys}
    |g_i(x)| + |\eta_i(x,t)| \le \sigma_i, \qquad i=1,2,3.
\end{equation}
In addition, to ensure practical implementability of the controller, 
the saturation nonlinearity and its associated mismatch term are also assumed to be bounded. 
That is, there exist positive constants $\xi_i$ such that
\begin{equation}\label{limit_sat}
   |\Delta u_i(t)| \le \lambda_i, \qquad i=1,2,3.
\end{equation}
\end{assum}

The following finite-time stability result will be used in the subsequent convergence analysis. 
\begin{lem}\label{lem:finite_time}
Let $V(t)$ be a continuous positive-definite function satisfying
\begin{equation*}
    \dot{V}(t)\le -qV^{\rho}(t),  \qquad q>0,\quad 0<\rho<1.
\end{equation*}
Then, $V(t)$ reaches zero in finite time, with the settling time bounded by
\begin{equation*}
    T_s\le t_0+ \frac{V^{1-\rho}(t_0)}{q(1-\rho)}.
\end{equation*}
\end{lem}

\subsection{The UFT-SMC control method}
As a first step toward stabilizing the chaotic Chua system, a flexible SMC strategy is adopted, and the following sliding surface is designed. 
The \textit{sliding surface} is defined as
\begin{equation}\label{SS}
 s_i(t) = |x_i(t)|+\mu_i x_i(t)+\int_0^t \Bigl(\gamma_1 |x_i(\tau)|^\omega+\gamma_2 |x_i(\tau)|^{p+\sign(x_i(\tau))}+\gamma_3\Bigr)\sign(x_i(\tau))\,d\tau,
\end{equation}
where 
$0<\omega<1$, $1<p<2$,
$\gamma_1$, $\gamma_2$, and $\gamma_3$ are positive constants and $\mu_i>1$.
On the sliding manifold $s_i(t)=0$, we have $\dot{s}_i(t)=0$, which gives
\begin{equation*}
   \dot{s}_i(t)=\frac{d}{dt}|x_i(t)|+\mu_i \dot{x}_i(t)
   +\Bigl(\gamma_1 |x_i(t)|^\omega +\gamma_2 |x_i(t)|^{p+\sign(x_i(t))} +\gamma_3\Bigr)\sign(x_i(t))=0.
\end{equation*}
Since $\frac{d}{dt}|x_i(t)|=\sign(x_i(t))\dot{x}_i(t)$, the above equation becomes
\begin{equation}\label{sdot}
  \dot{s}_i(t)=\sign(x_i(t))\dot{x}_i(t)+\mu_i \dot{x}_i(t)
  +\Bigl(\gamma_1 |x_i(t)|^\omega+\gamma_2 |x_i(t)|^{p+\sign(x_i(t))}+\gamma_3\Bigr) \sign(x_i(t))=0.
\end{equation}

Thus, the equivalent dynamics on the sliding surface is obtained as
\begin{equation}\label{SD}
\dot{x}_i(t) =
-\frac{\left(\gamma_1 |x_i(t)|^\omega+\gamma_2 |x_i(t)|^{p+\sign(x_i(t))}+\gamma_3\right)\sign(x_i(t))}{\sign(x_i(t))+\mu_i}.
\end{equation}

\begin{theorem}\label{thm:1}
Consider the equivalent sliding dynamics in \eqref{SD}. 
Under the proposed sliding surface \eqref{SS}, the system states $x_i(t)$ converge to the origin in finite time once the sliding manifold is reached.
\end{theorem} 
\begin{proof}
 Consider the Lyapunov candidate function
\begin{equation}\label{lya1}
     V_1(t) = \|X\|_1 = \sum_{i=1}^3 |x_i(t)|.
\end{equation}

Taking the time derivative, we get
\begin{equation*}
  \dot{V}_1(t) = \sum_{i=1}^{3} \frac{d}{dt} |x_i(t)| = \sum_{i=1}^{3} \sign(x_i(t)) \,\dot{x}_i(t).
\end{equation*}

Substituting the sliding dynamics \eqref{SD}, we have
\begin{equation*}
  \dot{V}_1(t) = -\sum_{i=1}^3 \frac{\gamma_1 |x_i(t)|^\omega
   +\gamma_2 |x_i(t)|^{p+\sign(x_i(t))} +\gamma_3}{\sign(x_i(t))+\mu_i}.
\end{equation*}
Since $\mu_i>1$, we define $\bar{\mu}=\max_{i=1,2,3} \{\mu_i+1\}$. Next, 
since $\sign(x_i(t))+\mu_i\le\mu_i+1\le\bar{\mu}$, we obtain
\begin{equation}\label{31}
\dot{V}_1(t)
\le
-\frac{1}{\bar{\mu}}
\sum_{i=1}^{3}
\Bigl(
\gamma_1|x_i(t)|^\omega
+\gamma_2|x_i(t)|^{p+\sign(x_i(t))}
+\gamma_3
\Bigr)
<0.
\end{equation}
Therefore, the sliding dynamics \eqref{SD} are stable.

According to \eqref{lya1} and \eqref{31}, we have
\begin{equation*}
\dot{V}_1(t) \le -\frac{\gamma_1}{\bar{\mu}}
\sum_{i=1}^{3}|x_i(t)|^\omega \le
-\frac{\gamma_1}{\bar{\mu}}
\biggl(\sum_{i=1}^{3}|x_i(t)|\biggr)^\omega
= - \frac{\gamma_1}{\bar{\mu}}V_1^\omega(t).
\end{equation*}
Hence, according to Lemma~\ref{lem:finite_time}, the states $x_i(t)$
converge to zero in finite time, with
\begin{equation*}
T_1
\le
\frac{\bar{\mu}V_1^{1-\omega}(0)}
{\gamma_1(1-\omega)}.
\end{equation*}
This completes the proof.
\end{proof}

\begin{figure}[!t]
\centering
\includegraphics[width=0.9\textwidth]{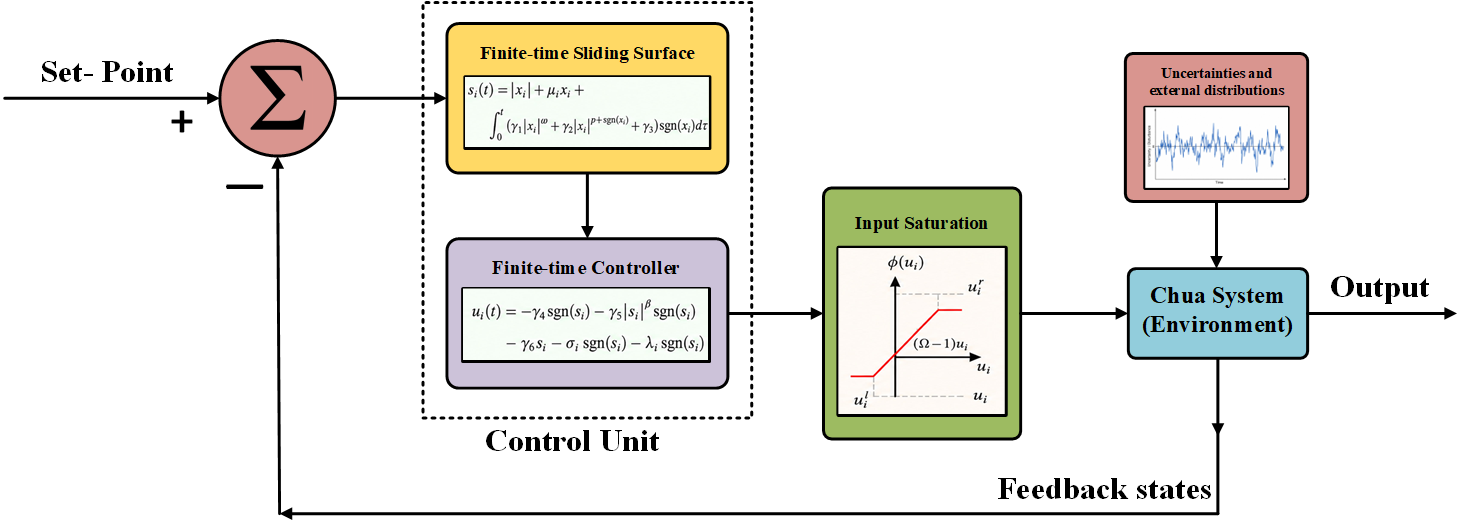}
\caption{Schematic of the proposed UFT-SMC framework for the Chua system with input saturation and lumped uncertainties.}
\label{Contro_Schematic}
\end{figure}

The control law is now developed to guarantee the convergence of the system trajectories to the sliding surface while ensuring robustness against system uncertainties and external disturbances. 
The overall structure of the proposed control scheme is illustrated in Figure~\ref{Contro_Schematic}.

For the controller design, the following control law is proposed
\begin{multline}\label{controller}
  u_i(t) = -\Bigg[(\sigma_i+\lambda_i)\sign(s_i(t))+\frac{1}{\sign(x_i(t))+\mu_i}
  \Big( (\gamma_1 |x_i(t)|^\omega+\gamma_2 |x_i(t)|^{p+\sign(x_i(t))}+\gamma_3)\sign(x_i(t)) \\
  +(\gamma_4 |s_i(t)|^\beta+\gamma_5 |s_i(t)|^{q+\sign(s_i(t))}+\gamma_6)\sign(s_i(t)) \Big) \Bigg].
\end{multline}
where, 
$0<\beta<1$, $1<q<2$,
$\gamma_4$, $\gamma_5$, and $\gamma_6$ are positive constants.

\begin{theorem}
Under the proposed control law \eqref{controller}, the trajectories of the Chua system \eqref{system} converge to the origin in finite time.
\end{theorem}
\begin{proof} 
Consider the Lyapunov candidate function
\begin{equation}\label{lya2}
    V_2(t)=\|S\|_1=\sum_{i=1}^{3}|s_i(t)|.
\end{equation}

Taking the time derivative yields
\begin{equation}\label{27}
    \dot{V}_2(t)
    =\sum_{i=1}^{3}\frac{d}{dt}|s_i(t)|
    =\sum_{i=1}^{3}\sign(s_i(t))\dot{s}_i(t).
\end{equation}

By substituting the sliding dynamics \eqref{sdot} into \eqref{27}, we obtain
\begin{equation*}
\dot{V}_2(t) = \sum_{i=1}^{3}\sign(s_i(t))
\Big[ \big(\sign(x_i(t))+\mu_i\big)\dot{x}_i(t)
+ \big( \gamma_1|x_i(t)|^\omega +\gamma_2|x_i(t)|^{p+\sign(x_i(t))} +\gamma_3 \big)\sign(x_i(t)) \Big].
\end{equation*}

Substituting the system dynamics \eqref{system} and the control law
\eqref{controller} into the above expression gives
\begin{multline*}
\dot{V}_2(t)
\le \sum_{i=1}^3 \big(\sign(x_i(t)) +\mu_i\big)
\Big[ \sign(s_i(t)) \big(g_i(x)+\eta_i(x,t)+\Delta u_i(t)\big) -(\sigma_i+\lambda_i) \Big] \\
 -\sum_{i=1}^3 \Big( \gamma_4|s_i(t)|^\beta +\gamma_5|s_i(t)|^{q+\sign(s_i(t))} +\gamma_6 \Big).
\end{multline*}

Since $\sign(x_i(t))+\mu_i>0$, applying the bounds
\eqref{limit_sys} and \eqref{limit_sat} yields
\begin{equation}\label{52}
\dot{V}_2(t)
\le
-\sum_{i=1}^{3}
\Big(
\gamma_4|s_i(t)|^\beta
+\gamma_5|s_i(t)|^{q+\sign(s_i(t))}
+\gamma_6
\Big)
<0.
\end{equation}

Therefore, the sliding manifold is asymptotically stable. Moreover, from
\eqref{lya2} and \eqref{52}, it follows that
\begin{equation}\label{V2_finite}
\dot{V}_2(t) \le -\gamma_4\sum_{i=1}^{3}
|s_i(t)|^\beta \le
-\gamma_4 \left( \sum_{i=1}^3 |s_i(t)|\right)^\beta =
-\gamma_4V_2^\beta(t) <0.
\end{equation}

According to Lemma~\ref{lem:finite_time}, the sliding variables reach
the sliding manifold in finite time, with the reaching time satisfying
\begin{equation}
T_s \le \frac{V_2^{1-\beta}(0)} {\gamma_4(1-\beta)}.
\end{equation}

Once the sliding manifold is reached, the system evolves according to
the equivalent sliding dynamics \eqref{SD}. Therefore, by
Theorem~\ref{thm:1}, the states $x_i(t)$ converge to the origin in
finite time. Hence, the closed-loop Chua system is finite-time stable,
which completes the proof.
\end{proof}

\begin{remark}
Following the above stability analysis, the flexibility of the proposed UFT-SMC can be further interpreted through the state-dependent exponents $p + \sign(x_i)$ and $q + \sign(s_i)$ in the  sliding surface and control law, respectively. 
By choosing  $0<\omega<1$, $1<p<2$, $0<\beta<1$, and $1<q<2$, the exponents switch  between $p+1>2$ and $p-1\in(0,1)$ for the state dynamics, and between $q+1>2$ and $q-1\in(0,1)$ for the sliding dynamics. 
This structure provides flexible shaping of the nonlinear convergence behavior without  requiring an explicit piecewise formulation. 
The outer sign functions preserve the restoring direction toward the equilibrium and the sliding  manifold, while the conditions $1<p<2$ and $1<q<2$ avoid negative-power  singularities near the origin.
Employing the same exponent-shaping  mechanism in both the state and sliding dynamics further yields a unified and structurally symmetric control design.
\end{remark}

%
\section{Stochastic Robustness Analysis}\label{sec:stochastic}
In practical implementations, model uncertainties, external perturbations, measurement errors, and actuator imperfections may vary randomly in time.
In this section, these effects are represented by a stochastic perturbation acting on the closed-loop sliding dynamics.
Rather than imposing moment conditions such as zero mean or bounded variance, we employ an almost-sure boundedness assumption. 
This is particularly natural for sliding-mode control, since its robustness properties are based on bounds on the magnitude of matched perturbations.

Let $(\Omega,\mathcal{F},\mathbb{P})$ be a probability space. For each
$i=1,2,3$, let
    $\zeta_i\colon[0,\infty)\times\Omega\to\mathbb{R}$
denote the residual stochastic perturbation in the $i$th sliding dynamics. 
We assume that its sample paths are measurable and locally integrable.
For convenience, define the set-valued sign function
\begin{equation}\label{eq:set_valued_sign}
\operatorname{Sgn}(s)=\begin{cases}
\{1\}, & s>0,\\[1mm]
[-1,1], & s=0,\\[1mm]
\{-1\}, & s<0.
\end{cases}
\end{equation}
This definition~\eqref{eq:set_valued_sign} allows the discontinuous sliding-mode dynamics to be
interpreted in the standard Filippov sense.
Furthermore, we define
\begin{equation}\label{eq:Psi_q}
\Psi_q(s)=\begin{cases}
|s|^{q+1}, & s>0,\\[1mm]
0, & s=0,\\[1mm]
|s|^{q-1}, & s<0,
\end{cases}
\qquad 1<q<2.
\end{equation}
Hence, for $s\neq 0$,
    $\Psi_q(s)=|s|^{q+\sign(s)}$.
The explicit definition in \eqref{eq:Psi_q} avoids the ambiguity of
the state-dependent exponent at $s=0$.

After substituting the proposed controller \eqref{controller} into the
sliding-surface dynamics and using the decomposition \eqref{sat_decomp},
define the \textit{residual term}
\begin{equation}\label{eq:residual_term}
r_i(t)=
\bigl(\sign(x_i(t))+\mu_i\bigr)
\bigl[g_i(x)+\eta_i(x,t)+\Delta u_i(t)\bigr].
\end{equation}

The stochastic perturbation $\zeta_i(t)$ is introduced as an additional
residual perturbation acting on the closed-loop sliding dynamics.
Accordingly, the perturbed reaching dynamics are represented by
\begin{equation}\label{eq:stochastic_sliding_dynamics}
    \begin{split}
\dot{s}_i(t)\in {}
r_i(t)&-\bigl(\sign(x_i(t))+\mu_i\bigr)
(\sigma_i+\lambda_i)\operatorname{Sgn}(s_i(t))\\
&\qquad - \bigl(\gamma_4|s_i(t)|^\beta +\gamma_5\Psi_q(s_i(t))
+\gamma_6 \bigr)\operatorname{Sgn}(s_i(t)) +\zeta_i(t).
 \end{split}
\end{equation}

\subsection{Almost-Sure Bounded Disturbances}
The following assumption specifies the stochastic disturbance class
considered in the analysis.

\begin{assum}[disturbance bound ]\label{ass:2}
For each $i=1,2,3$, there exists a deterministic constant
$\bar{\zeta}_i\ge0$ such that
\begin{equation}\label{eq:as_stochastic_bound}
\mathbb{P} \bigl(    |\zeta_i(t,\omega)|\le \bar{\zeta}_i
    \quad \text{for almost every } t\ge0 \bigr) =1.
\end{equation}
Moreover, the robust reaching gain satisfies
    $\gamma_6>\bar{\zeta}_i$,
    $i=1,2,3$.
\end{assum}
For later use, define the \textit{positive robustness margin}
    $\delta_i := \gamma_6-\bar{\zeta}_i >0$.

Assumption~\ref{ass:2} does not require the stochastic perturbation to have zero mean, finite variance, or any particular probability distribution.
The disturbance may therefore be biased, non-Gaussian, correlated in time, or generated by uncertain physical parameters, provided that its sample paths satisfy the bound \eqref{eq:as_stochastic_bound} almost surely.

\subsection{Almost-Sure Finite-Time Reaching}
We now show that the finite-time reaching property is preserved under the stochastic perturbations described above.
\begin{theorem}[Almost-sure finite-time reaching]\label{thm:3}
Consider the perturbed sliding dynamics \eqref{eq:stochastic_sliding_dynamics} and suppose that Assumption~\ref{ass:2} holds. 
Then, for every initial condition $s_i(t_0)\in\mathbb{R}$, the sliding variable $s_i(t)$ reaches the manifold
    $\mathcal{S}_i=\{s_i=0\}$
in finite time almost surely. More precisely, the hitting time
\begin{equation}\label{eq:hitting_time}
    T_{s,i} :=\inf \{ t\ge t_0\colon s_i(t)=0 \}
\end{equation}
satisfies
\begin{equation}\label{eq:stochastic_settling_time}
  T_{s,i}-t_0\le\frac{|s_i(t_0)|^{1-\beta}}{\gamma_4(1-\beta)}
\qquad \text{almost surely}.
\end{equation}
Moreover, after reaching, the sliding condition can be maintained
in the Filippov sense under the disturbance bound
\eqref{eq:as_stochastic_bound}.
\end{theorem}

\begin{proof}
Consider the Lyapunov function
    $V_i(t)=|s_i(t)|$.
For $s_i(t)\neq0$, the function $V_i$ is differentiable almost
everywhere and
    $\dot{V}_i(t) = \sign(s_i(t))\dot{s}_i(t)$.
Using \eqref{eq:stochastic_sliding_dynamics}, for $s_i(t)\neq0$ we obtain
\begin{equation*}
    \begin{split}
\dot{V}_i(t)
&\le \sign(s_i(t))r_i(t) -\bigl(\sign(x_i(t))+\mu_i\bigr)(\sigma_i+\lambda_i) \\
&\qquad -\gamma_4|s_i(t)|^\beta
-\gamma_5\Psi_q(s_i(t))
-\gamma_6
+\sign(s_i(t))\zeta_i(t).
  \end{split}
\end{equation*}

From the definition of $r_i(t)$ and the bounds
\eqref{limit_sys} and \eqref{limit_sat}, we have
\begin{equation*}
\sign(s_i(t))r_i(t) \le |r_i(t)| \le \bigl(\sign(x_i(t))+\mu_i\bigr)
(\sigma_i+\lambda_i).
\end{equation*}
Therefore,
\begin{equation*}
\dot{V}_i(t)
\le
-\gamma_4|s_i(t)|^\beta
-\gamma_5\Psi_q(s_i(t))
-\gamma_6
+|\zeta_i(t)|.
\end{equation*}
By Assumption~\ref{ass:2},
    $|\zeta_i(t)|\le\bar{\zeta}_i$
almost surely. Therefore,
\begin{equation}\label{eq:Vdot_final}
\begin{split}
\dot{V}_i(t)
&\le -\gamma_4 V_i(t)^\beta - \gamma_5\Psi_q(s_i(t)) - (\gamma_6-\bar{\zeta}_i)\\
&= -\gamma_4 V_i(t)^\beta -\gamma_5\Psi_q(s_i(t)) - \delta_i\\
&\le - \gamma_4 V_i(t)^\beta
\end{split}
\end{equation}
for almost every $t<T_{s,i}$ and almost every $\omega\in\Omega$.
Since $0<\beta<1$, inequality \eqref{eq:Vdot_final} can be integrated pathwise.
For $V_i(t)>0$,
\begin{equation*}
   \frac{d}{dt} V_i(t)^{1-\beta} = (1-\beta)V_i(t)^{-\beta} \dot{V}_i(t)
   \le -\gamma_4(1-\beta).
\end{equation*}
Now, integration from $t_0$ to $t<T_{s,i}$ yields
\begin{equation*} 
   V_i(t)^{1-\beta}\le V_i(t_0)^{1-\beta} - \gamma_4(1-\beta)(t-t_0).
\end{equation*}
Therefore, the right-hand side becomes zero no later than
\begin{equation*}
  t = t_0+\frac{V_i(t_0)^{1-\beta}}{\gamma_4(1-\beta)},
\end{equation*}
which proves \eqref{eq:stochastic_settling_time}.

It remains to consider the behavior on the manifold $s_i=0$.
At $s_i=0$, we have
$\Psi_q(0)=0$ and $\operatorname{Sgn}(0)=[-1,1]$.
Therefore, \eqref{eq:stochastic_sliding_dynamics} becomes
\begin{equation}\label{eq:filippov_at_zero}
\dot{s}_i \in r_i(t) - \Big[\bigl(\sign(x_i(t))+\mu_i\bigr)(\sigma_i+\lambda_i)
+\gamma_6 \Big][-1,1] +\zeta_i(t).
\end{equation}

From \eqref{eq:residual_term}, \eqref{limit_sys}, and
\eqref{limit_sat}, we have
\begin{equation*}
|r_i(t)| \le \bigl(\sign(x_i(t))+\mu_i\bigr) (\sigma_i+\lambda_i).
\end{equation*}
Hence, by Assumption~\ref{ass:2},
\begin{equation*}
    \begin{split}
|r_i(t)+\zeta_i(t)| &\le |r_i(t)|+|\zeta_i(t)|\\
&\le \bigl(\sign(x_i(t))+\mu_i\bigr)
(\sigma_i+\lambda_i) +\bar{\zeta}_i\\
&< \bigl(\sign(x_i(t))+\mu_i\bigr) (\sigma_i+\lambda_i) +\gamma_6
 \end{split}
\end{equation*}
almost surely.

Therefore, zero belongs to the right-hand side of \eqref{eq:filippov_at_zero}. 
Hence, an admissible Filippov selection
can satisfy $\dot{s}_i=0$ on the manifold, so the sliding condition
can be maintained after reaching.
Therefore, $s_i=0$ is invariant in the Filippov sense.

The argument holds for every disturbance realization belonging to the probability-one event defined by \eqref{eq:as_stochastic_bound}.
Hence the finite-time reaching result holds almost surely.
\end{proof}

\subsection{Finite-Time Stabilization of the State Variables}
Theorem~\ref{thm:3} establishes finite-time reaching of the sliding manifold under bounded stochastic perturbations. 
Once $s_i=0$ has been reached, the state dynamics are governed by the equivalent sliding dynamics derived in Section~\ref{sec:3}. 
Under the conditions of Theorem~\ref{thm:1}, these dynamics drive the states to the desired equilibrium in finite time.

Let $T_{s,i}$ denote the reaching time from Theorem~\ref{thm:3} and let $T_{x,i}$ denote the additional finite time required for the corresponding state $x_i$ to reach zero along the sliding manifold.
Then, the total convergence time satisfies
    $T_i \le T_{s,i}+T_{x,i}$
    almost surely.

In particular, applying the finite-time argument to each component of the sliding dynamics \eqref{SD} yields
\begin{equation*} 
   T_{x,i} \le\frac{(\mu_i+1)|x_i(T_{s,i})|^{1-\omega}}{\gamma_1(1-\omega)}.
\end{equation*}
Consequently,
\begin{equation*}
T_i \le t_0+\frac{|s_i(t_0)|^{1-\beta}}{\gamma_4(1-\beta)}
+\frac{(\mu_i+1)|x_i(T_{s,i})|^{1-\omega}}{\gamma_1(1-\omega)}
\qquad
\text{almost surely}.
\end{equation*}
Thus, bounded stochastic perturbations do not destroy the finite-time stabilization property, provided that the robust control margin $\gamma_6-\bar{\zeta}_i$ remains strictly positive.

\subsection{High-Probability Disturbance Bounds}
Assumption~\ref{ass:2} provides an almost-sure disturbance bound. 
In some applications, however, only a high-probability bound may be available.
The pathwise argument above can then be used conditionally on the corresponding high-probability event.

Suppose that, for a prescribed time horizon $T>t_0$, there exist $\bar{\zeta}_i>0$ and $\varepsilon_i\in(0,1)$ such that
\begin{equation}\label{eq:high_prob_bound}
\mathbb{P}\Bigl(\sup_{t_0\le t\le T}
    |\zeta_i(t)| \le \bar{\zeta}_i \Bigr) \ge 1-\varepsilon_i,
\end{equation}
and assume again that $\gamma_6>\bar{\zeta}_i$.
Define the deterministic upper bound
\begin{equation} 
\widehat{T}_{s,i} =t_0+\frac{|s_i(t_0)|^{1-\beta}}{\gamma_4(1-\beta)}.
\end{equation}
If
    $\widehat{T}_{s,i}\le T$,
then the pathwise proof of Theorem~\ref{thm:3} applies on the event appearing in
\eqref{eq:high_prob_bound}. Consequently,
we have the following corollary.

\begin{cor}[Finite-time reaching with prescribed
probability]\label{cor:1}
If the disturbance satisfies \eqref{eq:high_prob_bound},
$\gamma_6>\bar{\zeta}_i$, and $\widehat{T}_{s,i}\le T$, then
\begin{equation*}
   \mathbb{P}\bigl( T_{s,i}\le\widehat{T}_{s,i} \bigr) \ge1-\varepsilon_i.
\end{equation*}
Thus, any available high-probability bound for the stochastic
disturbance directly yields a corresponding probabilistic guarantee
for finite-time reaching.
\end{cor}

\section{Graph-Learning-Based Controller Parameter Optimization}\label{sec:5}
This section presents the graph-learning framework developed to optimize the controller parameters. 
The system information is first organized through graph representation and feature extraction, 
followed by the formulation of the learning objective and a GCN-based actor-critic scheme for parameter optimization.

\subsection{Learning-Based Parameter-Tuning Framework}
The design parameters of the proposed sliding-mode controller directly influence the transient response, reaching behavior, and control effort of the closed-loop system. 
To avoid extensive trial-and-error tuning, their selection is formulated as a constrained sequential decision problem and addressed through a \textit{deep reinforcement learning} (DRL) framework \cite{DRL_reference}.

At each learning step $k$, the agent receives a descriptor $\boldsymbol{\xi}_k$ of the closed-loop behavior and generates a parameter-selection action $\boldsymbol{\nu}_k$. 
The selected parameters are applied to the proposed controller, and the resulting system response provides the subsequent descriptor $\boldsymbol{\xi}_{k+1}$ together with a scalar reward $\mathcal{R}_k$. 
The parameter-selection policy is denoted by $\Pi_{\boldsymbol{\vartheta}_A}$, 
where $\boldsymbol{\vartheta}_A$ represents the trainable actor parameters.

The quality of a state-action pair under policy $\Pi$ is characterized through the discounted return
\begin{equation}\label{eq:return}
\mathcal{C}^{\Pi} (\boldsymbol{\xi}_k,\boldsymbol{\nu}_k)
= \mathbb{E}_{\Pi} \biggl[\sum_{m=0}^{T_L} \lambda_d^m \mathcal{R}_{k+m} \biggr],
\end{equation}
where $\lambda_d\in[0,1)$ is the discount factor and $T_L$ denotes the learning horizon. 
The objective is to learn a policy that improves the expected return while restricting the generated controller parameters to their theoretically admissible ranges. 
Therefore, the DRL mechanism is employed for parameter optimization rather than direct generation of the control input.
The overall workflow of the proposed graph-learning-based parameter optimization is illustrated in Figure~\ref{graph}, showing the interaction between the learning process and the UFT-SMC.

\begin{figure}[!t]
\centering
\includegraphics[width=0.9\textwidth]{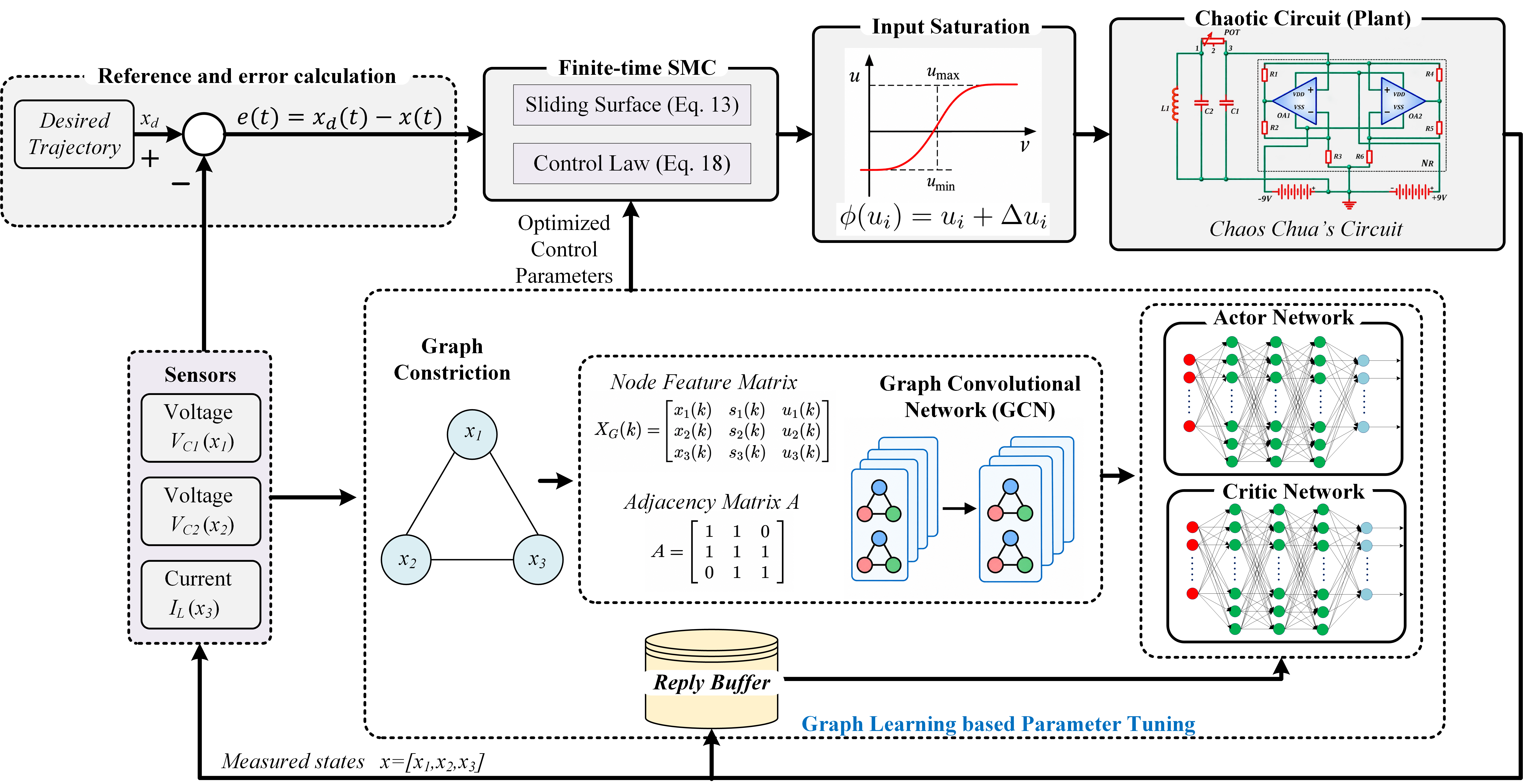}
\caption{Overall structure of the proposed graph-learning-based parameter optimization framework for the UFT-SMC.}
\label{graph}
\end{figure}

\subsection{Graph Representation and Feature Extraction}
A graph convolutional network (GCN) is employed to exploit the coupling structure of the Chua dynamics. 
The closed-loop system is represented by $\mathcal{G}_c=(\mathcal{V}_c,\mathcal{E}_c)$, where the three state channels constitute the graph nodes,
$\mathcal{V}_c=\{v_1,v_2,v_3\}$,
$v_i \leftrightarrow x_i$.
According to the dynamical interconnections, $x_1$ is coupled with $x_2$,
$x_2$ interacts with both $x_1$ and $x_3$, and $x_3$ is coupled with
$x_2$. The corresponding adjacency matrix reads
\begin{equation}\label{eq:adjacency}
\mathbf{A}_c=
\begin{bmatrix}
1&1&0\\
1&1&1\\
0&1&1
\end{bmatrix}.
\end{equation}
The diagonal entries account for the self-connections of the graph nodes.
Hence, the graph topology is directly obtained from the coupling structure of the plant dynamics.

To provide the learning network with both system and control information, the \textit{node-feature matrix} is 
\begin{equation}\label{eq:features}
\mathbf{X}_G(k)=
\begin{bmatrix}
x_1(k)&s_1(k)&u_1(k)\\
x_2(k)&s_2(k)&u_2(k)\\
x_3(k)&s_3(k)&u_3(k)
\end{bmatrix},
\end{equation}
where $x_i$, $s_i$, and $u_i$ denote the state, sliding variable, and control input of the $i$th channel, respectively.

Let $\mathbf{H}_G^{(0)}=\mathbf{X}_G$. 
The graph-convolution operation combines the node features according to the connectivity encoded by $\mathbf{A}_c$. 
The $\ell$th graph layer is expressed as
\begin{equation}\label{eq:gcn}
\mathbf{H}_G^{(\ell+1)}
= \sigma_g \Bigl( \mathbf{D}_c^{-1/2}\mathbf{A}_c \mathbf{D}_c^{-1/2}
\mathbf{H}_G^{(\ell)}\mathbf{W}_G^{(\ell)} + \mathbf{b}_G^{(\ell)} \Bigr),
\end{equation}
where $\mathbf{D}_c$ is the degree matrix associated with
$\mathbf{A}_c$, $\mathbf{W}_G^{(\ell)}$ and $\mathbf{b}_G^{(\ell)}$ are trainable parameters, and
$\sigma_g(\cdot)$ denotes the activation function. 
The output of the final graph layer is aggregated into a latent representation $\mathbf{z}_k$, which is subsequently supplied to the actor-critic architecture.

\subsection{Controller Parameterization and Learning Objective}
The action generated by the learning framework determines the adjustable parameters of the proposed controller. 
The optimization vector is defined as
\begin{equation}\label{eq:parameters}
\boldsymbol{\Theta}_c=
\begin{bmatrix}
   \gamma_1& \gamma_2& \gamma_3& \gamma_4& \gamma_5& \gamma_6& \mu_1& \mu_2& \mu_3
\end{bmatrix}^\top.
\end{equation}
Here, $\gamma_1$, $\gamma_2$, $\gamma_3$ and $\mu_i$ are associated with the sliding-surface design, whereas $\gamma_4$, $\gamma_5$, $\gamma_6$ determine the reaching dynamics of the controller. 
The exponent parameters $\omega$, $p$, $\beta$, and $q$ are kept fixed within their theoretically admissible ranges to reduce the search dimension and preserve the conditions employed in the convergence analysis.

Each optimized parameter is restricted to a prescribed interval,
$\theta_j^- \le \theta_j\le\theta_j^+$,
$\theta_j\in\boldsymbol{\Theta}_c$,
while the controller-design requirements impose
$\gamma_j>0$, $j=1,\ldots,6$,
   $\mu_i>1$, $i=1,2,3$.
Thus, the learning process explores only parameter combinations consistent with the theoretical controller design.
The \textit{learning reward} accounts for state stabilization, convergence toward the sliding manifold, and control effort:
\begin{equation}\label{eq:reward}
  \mathcal{R}_k= - \bigl[ c_x \|\mathbf{x}(k)\|_2^2
  + c_s \|\mathbf{s}(k)\|_2^2 + c_u \|\mathbf{u}(k)\|_2^2 \bigr],
\end{equation}
where $c_x$, $c_s$, $c_u>0$ are weighting coefficients. 
This reward \eqref{eq:reward} is used during training, while the conventional integral performance indices are reserved for the quantitative evaluation of the final closed-loop responses.

\subsection{GCN-Based Actor-Critic Parameter Optimization}
The latent graph representation $\mathbf{z}_k$ is provided to the actor, which generates a candidate controller-parameter vector. 
To guarantee feasible parameter values, the actor output is passed through a bound-preserving mapping $\mathcal{M}_{\Theta}(\cdot)$:
\begin{equation}\label{eq:actor}
  \boldsymbol{\Theta}_{c,k} = \mathcal{M}_{\Theta}
  \bigl[\mathcal{A} (\mathbf{z}_k;\boldsymbol{\vartheta}_A) \bigr],
\end{equation}
where $\mathcal{A}(\cdot)$ denotes the actor and $\boldsymbol{\vartheta}_A$ represents its trainable parameters. 
The resulting $\boldsymbol{\Theta}_{c,k}$ is employed in the proposed sliding surface and control law, 
and the corresponding closed-loop response is used to evaluate the selected parameters.

The critic
$\mathcal{C} (\mathbf{z}_k, \boldsymbol{\Theta}_{c,k}; \boldsymbol{\vartheta}_C )$
estimates the long-term quality of the actor decision, where $\boldsymbol{\vartheta}_C$ denotes the critic parameters. 
Target actor and critic networks, denoted by $\mathcal{A}^-$ and $\mathcal{C}^-$, are employed to improve learning stability. 
For the $j$th sample extracted from the replay memory, the \textit{target value} is calculated as
\begin{equation}\label{eq:target}
\mathcal{Y}_j = \mathcal{R}_j + \lambda_d
\mathcal{C}^- \bigl( \mathbf{z}_{j+1},
\mathcal{M}_{\Theta} [ \mathcal{A}^- (\mathbf{z}_{j+1}) ] \bigr).
\end{equation}

The critic parameters are updated by minimizing the temporal-difference loss
\begin{equation}\label{eq:critic_loss}
\mathcal{J}_C =\frac{1}{N_B}
\sum_{j=1}^{N_B} \bigl[
\mathcal{Y}_j- \mathcal{C}( \mathbf{z}_j, \boldsymbol{\Theta}_{c,j}; \boldsymbol{\vartheta}_C) \bigr]^2,
\end{equation}
where $N_B$ denotes the mini-batch size.
The actor parameters are updated using the deterministic policy gradient
\begin{equation}\label{eq:actor_gradient}
\nabla_{\boldsymbol{\vartheta}_A}\mathcal{J}_A
\approx \frac{1}{N_B} \sum_{j=1}^{N_B} \nabla_{\boldsymbol{\Theta}_c}
\mathcal{C}(\mathbf{z}_j, \boldsymbol{\Theta}_c; \boldsymbol{\vartheta}_C )
\big|_{ \boldsymbol{\Theta}_c=
\mathcal{M}_{\Theta}
[ \mathcal{A}(\mathbf{z}_j;\boldsymbol{\vartheta}_A)] }
\times \nabla_{\boldsymbol{\vartheta}_A}
\mathcal{M}_{\Theta}[ \mathcal{A} (\mathbf{z}_j;\boldsymbol{\vartheta}_A) ].
\end{equation}

The target-network parameters are softly updated according to
\begin{equation}\label{eq:soft_update}
\boldsymbol{\vartheta}_{A,C}^-
\leftarrow (1-\varepsilon_t) \boldsymbol{\vartheta}_{A,C}^-
+ \varepsilon_t \boldsymbol{\vartheta}_{A,C},
\qquad 0<\varepsilon_t\ll1.
\end{equation}

During training, exploratory perturbations are introduced to examine different regions of the admissible parameter space, and the resulting transitions are stored in replay memory for network updates. 
After convergence, the best admissible parameter vector $\boldsymbol{\Theta}_c^{*}$ is retained and employed in the proposed sliding-mode controller for the final closed-loop implementation.


\section{Numerical Simulations}\label{sec:6}
The effectiveness of the proposed approach is examined under the considered simulation setup, with further comparisons providing a quantitative assessment of its control performance.
\subsection{Simulation Setup}
Numerical simulations are carried out in MATLAB to evaluate the finite-time stabilization performance of the proposed UFT-SMC method \eqref{controller} for the uncertain Chua chaotic system 
\eqref{system}.
The initial conditions are selected as $x_1(0)=0.1$, $x_2(0)=0.05$, and $x_3(0)=0$.
The controller is activated at $t=0.1$\,s, while the control inputs and sliding variables are set to zero prior to this instant. 
The controller parameters are determined using the GCN-based actor-critic optimization framework as $\sigma_1=4.2$, $\sigma_2=1.3$, $\sigma_3=0.001$, 
$\lambda_1=\lambda_2=\lambda_3=0.2$, 
$\gamma_1=2.43$, $\gamma_2=1.9$, $\gamma_3=2.46$, $\gamma_4=1.72$,
$\gamma_5=3.11$, $\gamma_6=2.77$, $\mu_1=\mu_2=\mu_2=1.4$, 
$p=1.65$, $q=1.87$, and $\omega=\beta=0.8$.

Figure~\ref{fig:state_response} (a-c) presents the closed-loop state responses of the stabilized Chua chaotic system \eqref{system}.
Before the controller is activated, the trajectories retain the oscillatory behavior of the uncontrolled chaotic system. 
Once the control action is applied at $t=0.1$\,s, $x_1(t)$, $x_2(t)$, and $x_3(t)$ are rapidly driven toward the origin and reach a small neighborhood of the equilibrium within a finite time. 
The subsequent responses remain close to zero, demonstrating successful suppression of the chaotic dynamics and the finite-time stabilization capability of the proposed controller.

\begin{figure}[H]
\centering
\includegraphics[width=0.9\textwidth]{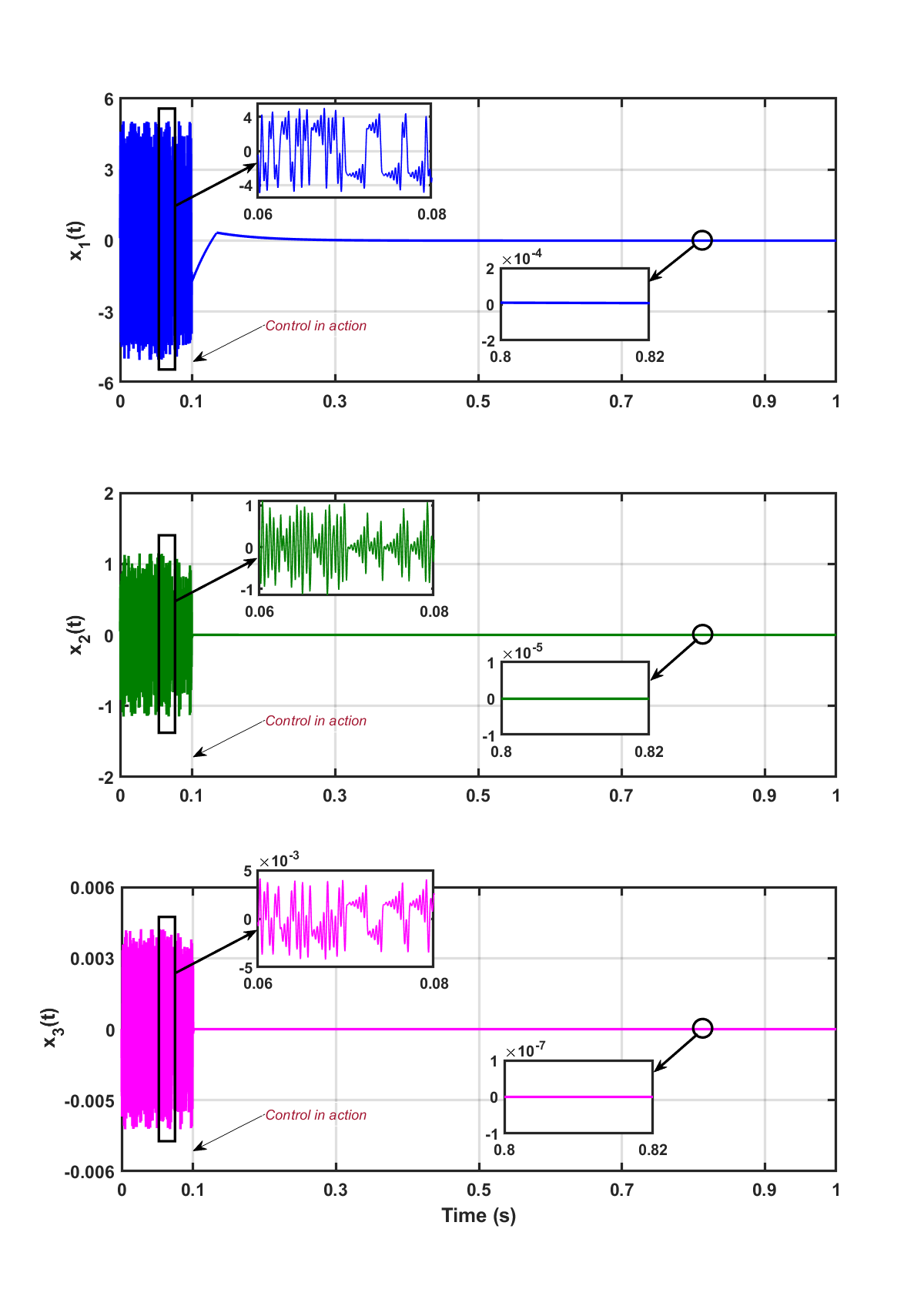}
\caption{Closed-loop evolution of the uncertain Chua system under the proposed finite-time control scheme: 
(a) $x_1(t)$; (b) $x_2(t)$; and (c) $x_3(t)$.}
\label{fig:state_response}
\end{figure}

Moreover, the reaching behaviour of the sliding variables is illustrated in Figure~\ref{fig:sliding_response} (a-c). 
Following controller activation, $s_1(t)$, $s_2(t)$, and $s_3(t)$ rapidly decay toward zero and reach the sliding manifold within a short finite-time interval.
Their Euclidean norm, $\|\mathbf{s}(t)\|_2=\sqrt{s_1^2(t)+s_2^2(t)+s_3^2(t)}$, provides a compact representation of this behavior. 
As observed in Figure~\ref{fig:sliding_response} (d), the norm rapidly approaches zero and remains close to the origin thereafter. 
These responses are consistent with the finite-time reaching property established in the theoretical analysis.

\begin{figure}[H]
\centering
\includegraphics[width=0.9\textwidth]{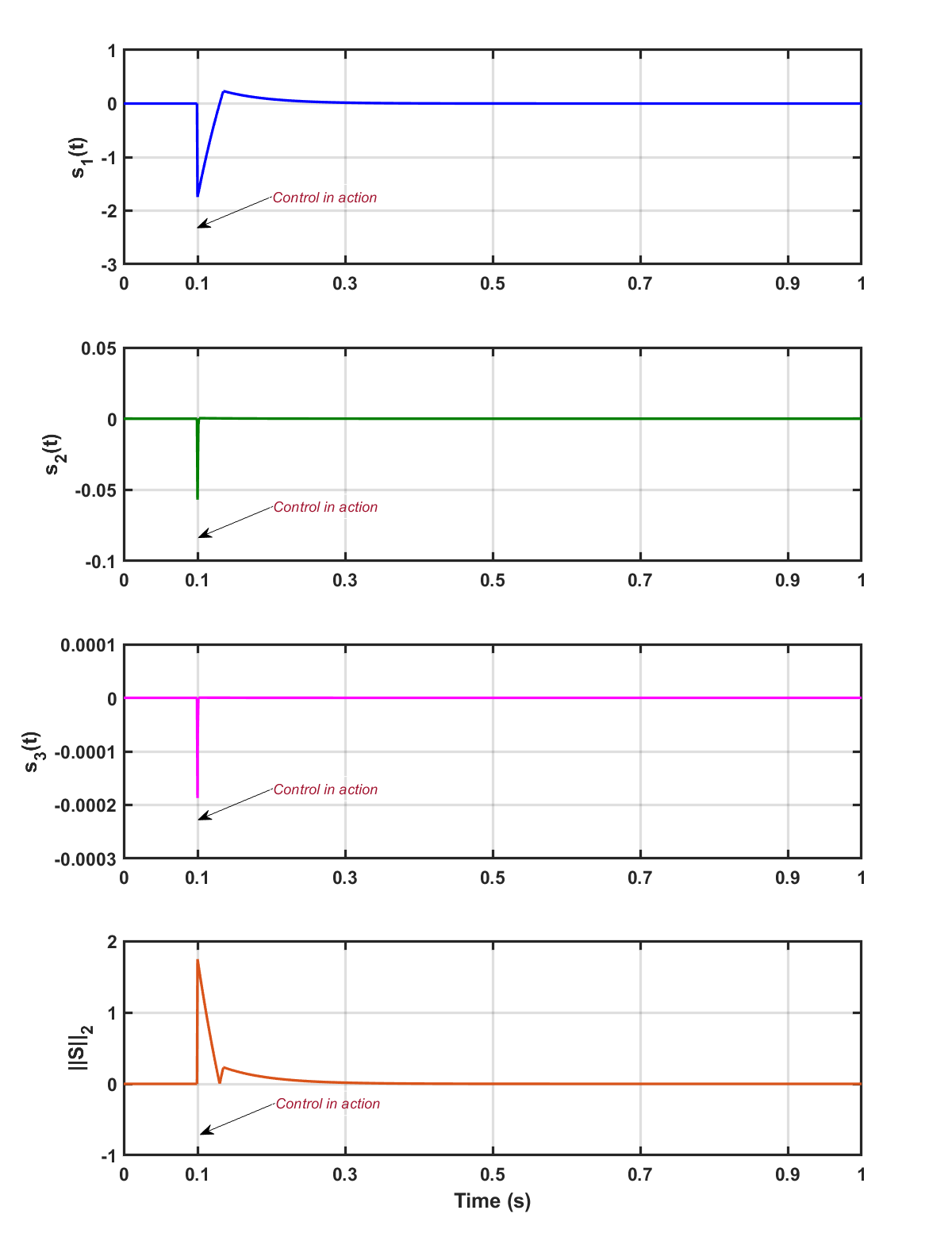}
\caption{Finite-time reaching characteristics of the proposed sliding-mode
scheme: (a) $s_1(t)$; (b) $s_2(t)$; (c) $s_3(t)$; and
(d) $\|\mathbf{s}(t)\|_2$.}
\label{fig:sliding_response}
\end{figure}

Furthermore, Figure~\ref{fig:control_response} (a-c) depicts the control signals during the stabilization process. 
Immediately after activation, the proposed controller provides the control action required to suppress the chaotic motion while respecting the prescribed saturation limits. 
The control signals remain bounded by these limits during the transient response and decrease rapidly as the states approach the equilibrium. 
Their Euclidean norm, $\|\mathbf{u}(t)\|_2=\sqrt{u_1^2(t)+u_2^2(t)+u_3^2(t)}$, further shows the reduction in the required control effort after stabilization.
As seen in Figure~\ref{fig:control_response}(d), only a small control action is required once the finite-time transient has elapsed. 
Hence, the saturation mechanism effectively constrains the control amplitudes without preventing the rapid stabilization of the chaotic system.

\begin{figure}[H]
\centering
\includegraphics[width=0.9\textwidth]{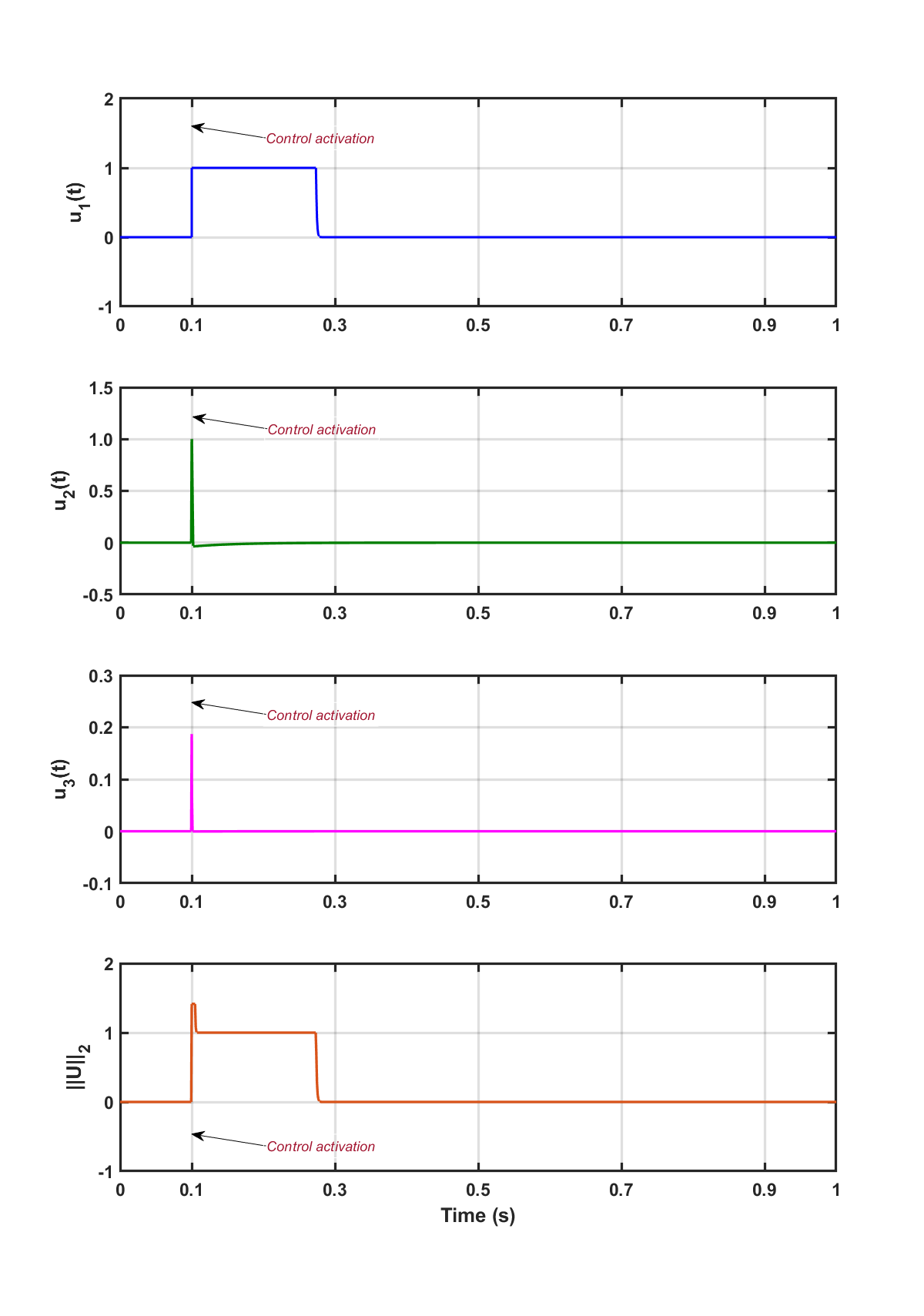}
\caption{Control activity during the finite-time stabilization process:
(a) $u_1(t)$; (b) $u_2(t)$; (c) $u_3(t)$; and
(d) $\|\mathbf{u}(t)\|_2$.}
\label{fig:control_response}
\end{figure}

Taken together, the state, sliding-variable, and control responses confirm that the proposed UFT-SMC \eqref{controller} scheme suppresses the chaotic dynamics within a finite time, rapidly establishes the sliding motion, and maintains the required control action within the prescribed saturation limits.

\subsection{Comparative and Performance Evaluation}
The control performance of the proposed UFT-SMC is compared with the method reported in \cite{WU20}, with the corresponding state responses shown in Figure~\ref{fig:comparison}. 
Both controllers suppress the chaotic oscillations after activation at $t=0.1$\,s. 
The UFT-SMC, however, regulates the states more rapidly, particularly during the transient period following control activation. 
The enlarged regions also reveal a clear difference near the equilibrium point, where the proposed controller maintains the states closer to zero with smaller residual errors than the existing method.

\begin{figure}[H]
\centering
\includegraphics[width=0.9\textwidth]{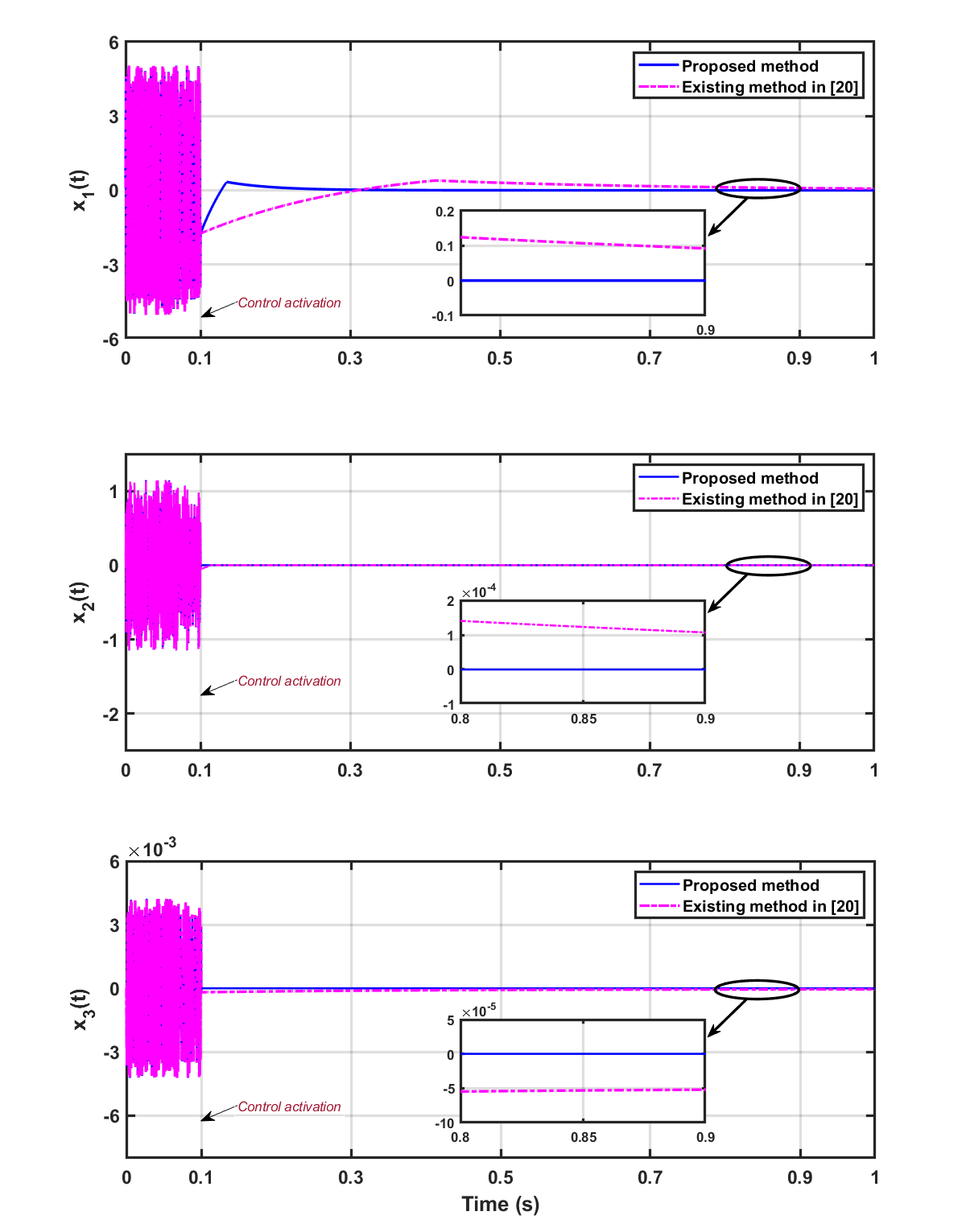}
\caption{Comparison of the state responses obtained using the proposed UFT-SMC and the existing method in \cite{WU20}.}
\label{fig:comparison}
\end{figure}

The comparison is further quantified using four performance indices:  the integral absolute error (IAE), 
integral squared error (ISE), 
integral time-weighted absolute error (ITAE),
and integral time-weighted squared error (ITSE), defined as
\begin{equation}\label{eq:performance_indices}
\begin{aligned}
  \mathrm{IAE} &= \int_0^T \sum_{i=1}^{3}|x_i(t)|\,dt, &
  \mathrm{ISE} &= \int_0^T \sum_{i=1}^{3}x_i^2(t)\,dt,\\
  \mathrm{ITAE} &= \int_0^T t\sum_{i=1}^{3}|x_i(t)|\,dt, &
  \mathrm{ITSE} &= \int_0^T t\sum_{i=1}^{3}x_i^2(t)\,dt.
\end{aligned}
\end{equation}
Figure~\ref{fig:performance}(a)--(b) compares the four performance indices obtained with the UFT-SMC and the controller in \cite{WU20}. 
The proposed controller achieves lower values in every case, with IAE and ISE decreasing from $0.2786$ to $0.1968$ and from $0.1920$ to $0.1308$, respectively. 
These reductions indicate that the UFT-SMC limits both the accumulated absolute state deviations and their squared magnitudes more effectively over the evaluation interval.

The \textit{relative improvement} is quantified as
\begin{equation*}
\mathrm{Improvement}(\%)=
\frac{J_{Ref. \text{\cite{WU20}}
}-J_{\mathrm{UFT\text{-}SMC}}}
     {J_{Ref. \text{\cite{WU20}}}}\times 100,
\end{equation*}
where $J$ denotes the corresponding performance index.
As summarized in Figure~\ref{fig:performance}(c), the
reductions in IAE, ISE, ITAE, and ITSE are
$29.36\,\%$, $31.88\,\%$, $44.56\,\%$, and $40.38\,\%$,
respectively. 
The larger relative gains in ITAE and ITSE are particularly relevant because these indices assign greater weight to deviations occurring at later times.
Taken together, the results support the superior regulation performance of the UFT-SMC in this comparison,  
with smaller accumulated state deviations and more effective suppression of errors persisting at later times.

\begin{figure}[htb]
    \centering
    \includegraphics[width=0.9\textwidth]{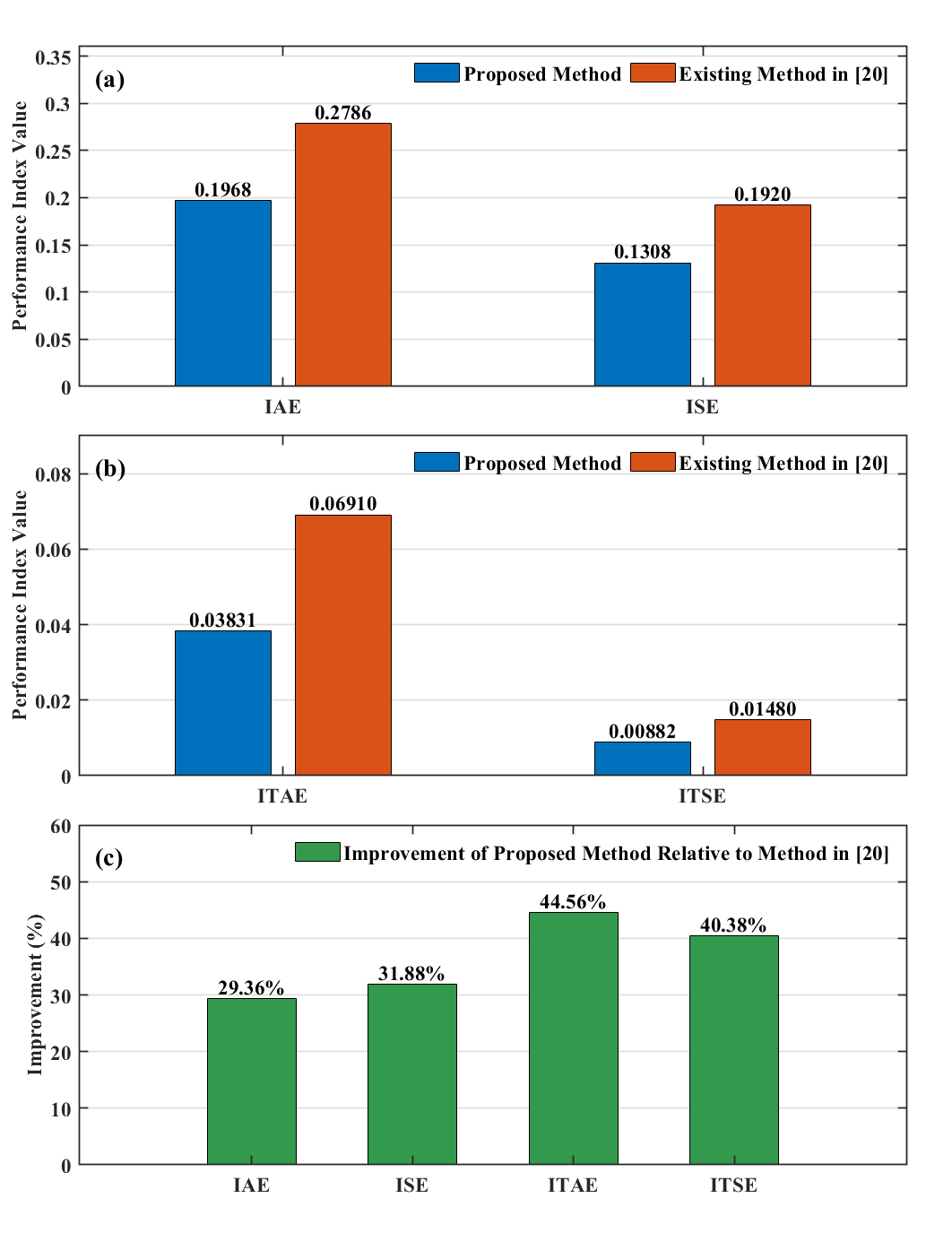}
    \caption{Performance comparison of the proposed UFNTSMC, and the existing method in \cite{WU20}: (a) IAE and ISE, (b) ITAE and ITSE, 
    and (c) percentage improvement relative to the existing method in \cite{WU20}.}
    \label{fig:performance}
\end{figure}
%

\section{Conclusion}\label{sec:7}
The Chua circuit remains a representative nonlinear system for studying chaotic dynamics and assessing control strategies under physically meaningful operating conditions. 
In this work, a graph-learning-assisted UFT-SMC was developed for its finite-time stabilization in the presence of input saturation, uncertainties, and external disturbances. 
The controller suppresses the chaotic oscillations and drives the circuit states to the desired equilibrium within a short transient period, while the sliding variables converge to zero and the control inputs remain within their prescribed bounds.
The graph-learning scheme was used to tune the controller parameters according to their influence on the closed-loop response, reducing the need for manual parameter adjustment. 
The resulting design showed faster regulation, smaller residual errors, and better integral performance indices than the existing method considered for comparison. 
These results show that graph-based learning can be effectively incorporated into nonlinear control design as a practical parameter-tuning tool without altering the underlying control structure. 
Future work will consider experimental validation of the proposed approach and its application to other nonlinear and fractional-order dynamical systems.

\clearpage
\section*{Declaration of Interests}
The authors declare that they have no known competing financial interests or personal relationships that could have appeared to influence the work reported in this paper.

\section*{Declaration on the Use of Artificial Intelligence}

The authors used ChatGPT (OpenAI) for language editing and improving the clarity of the manuscript. 
The authors are fully responsible for the scientific content.
\printbibliography

@book{rasband2015chaotic,
  title={Chaotic Dynamics of Nonlinear Systems},
  author={Rasband, S.N.},
  isbn={9780486805771},
  series={Dover Books on Physics},
  url={},
  year={2015},
  publisher={Dover Publications}
}

@article{jaiswal2024,
  title={Cluster synchronization of networks of chaotic systems: A comprehensive review of theory and applications},
  author={Jaiswal, Akshay Kumar and Sharma, Bharat Bhushan},
  journal={IETE Technical Review},
  volume={41},
  number={5},
  pages={537--556},
  doi={10.1080/02564602.2024.2341076},
  year={2024},
  publisher={Taylor \& Francis}
}

@article{kennedy1992a,
  author = {Kennedy, M. P.},
  title = {Robust OP Amp Realization of Chua's Circuit},
  format = {article},
  journal = {Frequenz},
  volume = {46},
  number = {3-4},
  pages = {66-80},
  year = {1992},
  issn = {21916349, 00161136},
  doi = {10.1515/FREQ.1992.46.3-4.66}
}

@article{chua1994a,
author = {Chua, Leon O.},
title = {Chua's circuit 10 years later},
journal = {International Journal of Circuit Theory and Applications},
volume = {22},
number = {4},
pages = {279-305},
doi = {10.1002/cta.4490220404},
url = {},
eprint = {},
year = {1994}
}

@ARTICLE{Truong,
  author={Truong, Hoai Vu Anh and Nam, Seokho and Kim, Sejin and Kim, Youngwan and Chung, Wan Kyun},
  journal={IEEE Transactions on Automation Science and Engineering}, 
  title={Backstepping-Sliding-Mode-Based Neural Network Control for Electro-Hydraulic Actuator Subject to Completely Unknown System Dynamics}, 
  year={2024},
  volume={21},
  number={4},
  pages={6202-6216},
  doi={10.1109/TASE.2023.3323148}}

@ARTICLE{Agha1,
  author={Aghababa, Mohammad Pourmahmood},
  journal={IEEE/ASME Transactions on Mechatronics}, 
  title={Sliding-Mode Control Composite With Disturbance Observer for Tracking Control of Mismatched Uncertain nDoF Nonlinear Systems}, 
  year={2018},
  volume={23},
  number={1},
  pages={482-490},
  doi={10.1109/TMECH.2017.2781624}}

@article{agha2,
  title={Stabilisation of unknown fractional-order chaotic systems: an adaptive switching control strategy with application to power systems},
  author={Aghababa, Mohammad Pourmahmood and Haghighi, Ahmad Reza and Roohi, Majid},
  journal={IET generation, transmission \& distribution},
  volume={9},
  number={14},
  pages={1883--1893},
  year={2015},
  publisher={Wiley Online Library}
}

@article{Azioune,
author = {Azioune, Mourad and Abdelouahab, Mohammed Salah and Lozi, Ren{\'e}},
title = {Bifurcation Analysis of a Cournot Triopoly Game with Bounded Rationality and Chaos Control Via the OGY Method},
journal = {International Journal of Bifurcation and Chaos},
volume = {35},
number = {08},
pages = {2530019},
year = {2025},
doi = {10.1142/S0218127425300198},
URL = {},
eprint = {}
}

@book{SMC,
  title={Sliding mode control and observation},
  author={Shtessel, Yuri and Edwards, Christopher and Fridman, Leonid and Levant, Arie and others},
  volume={10},
  year={2014},
  publisher={Springer}
}

@ARTICLE{10403641,
  author={Zhang, Zhang and Yang, Xiaodong and Wang, Weiyu and Chen, Kaiwen and Cheung, Norbert Chow and Pan, Jianfei},
  journal={IEEE Transactions on Industrial Electronics}, 
  title={Enhanced Sliding Mode Control for PMSM Speed Drive Systems Using a Novel Adaptive Sliding Mode Reaching Law Based on Exponential Function}, 
  year={2024},
  volume={71},
  number={10},
  pages={11978-11988},
  doi={10.1109/TIE.2023.3347845}}

@article{gedef,
  title={An improved trajectory tracking control of quadcopter using a novel sliding mode control with fuzzy PID surface},
  author={Gedefaw, Elisabeth Andarge and Abdissa, Chala Merga and Lemma, Lebsework Negash},
  journal={PloS one},
  volume={19},
  number={11},
  pages={e0308997},
  year={2024},
  publisher={Public Library of Science San Francisco, CA USA},
  doi={10.1371/journal.pone.0308997}
}

@article{SAK,
title = {Adaptive sliding mode control of multilevel grid-connected inverters using reinforcement learning for enhanced LVRT performance},
journal = {Electric Power Systems Research},
volume = {253},
pages = {112428},
year = {2026},
issn = {0378-7796},
doi = {10.1016/j.epsr.2025.112428},
url = {},
author = {Tamil Selvi Sakthivel and Prakash Ragupathy and Nagarajan Chinnadurai},
}

@article{FENG,
title = {A new adaptive sliding mode controller based on the RBF neural network for an electro-hydraulic servo system},
journal = {ISA Transactions},
volume = {129},
pages = {472-484},
year = {2022},
issn = {0019-0578},
doi = {10.1016/j.isatra.2021.12.044},
url = {},
author = {Hao Feng and Qianyu Song and Shoulei Ma and Wei Ma and Chenbo Yin and Donghui Cao and Hongfu Yu}
}

@article{meny,
  title={Optimized backstepping fuzzy sliding mode controller for trajectory tracking of mobile manipulator},
  author={Menyechel Eneyew, Geta and Ayalew Asfaw, Wubshet and Merga Abdissa, Chala},
  journal={Engineering Reports},
  volume={7},
  number={7},
  pages={e70269},
  year={2025},
  publisher={Wiley Online Library}
}

@article{ram,
  title={Fuzzy logic-enhanced sliding mode control of Belousov--Zhabotinsky reaction dynamics},
  author={Ramezani Pargami, Negin and Ghorbanian, Sohrab Ali and Fatoorehchi, Hooman},
  journal={The Canadian Journal of Chemical Engineering},
  volume={104},
  number={2},
  pages={781--795},
  year={2026},
  publisher={Wiley Online Library}
}

@article{lan,
  title={Data-Driven Sliding Mode Control for Partially Unknown Nonlinear Systems},
  author={Lan, Jianglin and Zhao, Xianxian and Sun, Congcong},
  journal={International Journal of Robust and Nonlinear Control},
  volume={36},
  number={2},
  pages={629--638},
  year={2026},
  publisher={Wiley Online Library}
}

@article{curran1997absolute,
  title={Absolute stability theory and the synchronization problem},
  author={Curran, PF and Chua, LO},
  journal={International Journal of Bifurcation and Chaos},
  volume={7},
  number={06},
  pages={1375--1382},
  year={1997},
  publisher={World Scientific}
}

@article{fradkov2005control,
  title={Control of chaos: Methods and applications in engineering},
  author={Fradkov, Alexander L and Evans, Robin J},
  journal={Annual Reviews in Control},
  volume={29},
  number={1},
  pages={33--56},
  year={2005},
  publisher={Elsevier}
}

@article{DRL_reference,

title = {Multi-agent reinforcement learning with graph convolutional neural networks for optimal bidding strategies of generation units in electricity markets},
journal = {Expert Systems with Applications},
volume = {225},
pages = {120010},
year = {2023},
issn = {0957-4174},
doi = {10.1016/j.eswa.2023.120010},
url = {},
author = {Pegah Rokhforoz and Mina Montazeri and Olga Fink},

}

@article{WU20,
title = {Fixed-time composite neural learning control of state-constrained nonlinear uncertain systems},
journal = {Journal of the Franklin Institute},
volume = {358},
number = {16},
pages = {8138-8168},
year = {2021},
issn = {0016-0032},
doi = {10.1016/j.jfranklin.2021.08.013},
url = {},
author = {Zhonghua Wu and Shuaipeng Zheng and Yunji Zhao and Bojun Liu and Junkang Ni},

}

\end{document}